\documentclass[11pt, leqno]{book}

\usepackage{amssymb, amsmath, graphicx, setspace, amsthm, fancyhdr, makeidx}
\usepackage[left=1in,top=1.15in,right=1in,bottom=1.15in]{geometry}

\fancypagestyle{plain}{%
\fancyhead{} 
} 

\swapnumbers
\theoremstyle{definition}
\newtheorem{definition}{Definition}[section]
\newtheorem{remark}[definition]{Remark}
\newtheorem{example}[definition]{Example}
\newtheorem{examples}[definition]{Examples}
\newtheorem{warning}[definition]{Warning}
\theoremstyle{plain}
\newtheorem{theorem}[definition]{Theorem}
\newtheorem{lemma}[definition]{Lemma}
\newtheorem{corollary}[definition]{Corollary}

\renewcommand{\iff}{{\em\, i\,f\,{f}\quad}}
\newcommand{\bra}{\ensuremath{\langle}}
\newcommand{\ket}{\ensuremath{\rangle}}

\newcommand{\card}{\ensuremath{{\rm{card}}}}

\newcommand{\dom}{\ensuremath{{\rm{dom}}}}
\newcommand{\ran}{\ensuremath{{\rm{ran}}}}
\newcommand{\st}{\ensuremath{{{\rm{st}}}}}

\newcommand{\scf}{\mathcal{F}}
\newcommand{\sca}{\mathcal{A}}
\newcommand{\scu}{\mathcal{U}}
\newcommand{\scv}{\mathcal{V}}
\newcommand{\sci}{\mathcal{I}}
\newcommand{\scl}{\mathcal{L}}

\newcommand{\scs}{\mathcal{S}}

\newcommand{\poweri}{\mathcal{P}(I)}
\newcommand{\powerplus}{\mathcal{P}(\mathbb{R}_+)}
\newcommand{\rplus}{\mathbb{R}_+}
\newcommand{\crplus}{\mathbb{C}^{\mathbb{R}_+}}
\newcommand{\varep}{\varepsilon}
\newcommand{\mathr}{\mathbb{R}}
\newcommand{\mathn}{\mathbb{N}}
\newcommand{\mathq}{\mathbb{Q}}
\newcommand{\subeq}{\subseteq}
\newcommand{\mathc}{\mathbb{C}}
\newcommand{\polycstar}{\starc[x]}
\newcommand{\aep}{a_\varep}
\newcommand{\bep}{b_\varep}
\newcommand{\cep}{c_\varep}
\newcommand{\zep}{z_\varep}
\newcommand{\xep}{x_\varep}
\newcommand{\xdlt}{x_\delta}
\newcommand{\yep}{y_\varep}
\newcommand{\sep}{s_\varep}
\newcommand{\rep}{r_\varep}
\newcommand{\Sep}{\mathcal{S}_\varep}
\renewcommand{\star}{{^{*}}}
\newcommand{\starf}{{^{*}f}}
\newcommand{\starg}{{^{*}g}}

\newcommand{\starc}{{^{*}\mathc}}
\newcommand{\starrplus}{{^{*}\rplus}}
\newcommand{\starx}{{^{*}X}}
\newcommand{\starn}{{^{*}\mathn}}
\newcommand{\starr}{{^{*}\mathr}}

\newcommand{\stars}{{^{*}S}}
\newcommand{\stara}{{^{*}A}}
\newcommand{\starb}{{^{*}B}}
\renewcommand{\iff}{{\em\, i\,f\,{f}\quad}}

\newcommand{\anep}{a_{n\varep}}
\newcommand{\Aep}{A_{\varep}}
\newcommand{\Xep}{X_{\varep}}
\newcommand{\Xdlt}{X_{\delta}}
\newcommand{\bnep}{b_{n\varep}}
\newcommand{\amep}{a_{m\varep}}
\newcommand{\bmep}{b_{m\varep}}
\newcommand{\scistarc}{\sci(\starc)}
\newcommand{\scfstarc}{\scf(\starc)}
\newcommand{\epn}{\varep_{0}}
\newcommand{\nball}{B(s, \frac{1}{n})}
\newcommand{\starball}{{^*B(s, \frac{1}{n})}}
\makeindex

\begin{document}
\pagenumbering{roman}

\pagestyle{empty}
\Large
\noindent \textbf{Reduction of the Number of Quantifiers\\in Real Analysis through Infinitesimals}\\

\normalsize
\vspace{6.6in}
\noindent Ray Cavalcante\\             
California Polytechnic State University\\
San Luis Obispo, California\\
Adviser: Todor D. Todorov
\newpage

\pagestyle{fancy}
\begin{center}
\textbf{Authorization for Reproduction}\vspace{0.5in}
\end{center}

\begin{quote}
I grant permission for the reproduction of this thesis in its entirety or any of its parts, without further authorization from me.\vspace{0.5in}
\end{quote}

\begin{center}
\begin{tabular}{lcl}
\underline{\hspace{12em}} & & \underline{\hspace{12em}} \\
Signature & \hspace{6em} & Date\\
\end{tabular}
\end{center}
\newpage

\begin{center}
\textbf{Approval}\vspace{0.5in}
\end{center}

\begin{center}
\begin{tabular}{rl}
	\emph{Title:} & Reduction of the Number of Quantifiers in Real Analysis through Infinitesimals\\

	\emph{Name:} & Raymond Cavalcante\\

	\emph{Date:} & December 4, 2007
\end{tabular}\vspace{0.75in}
\end{center}

\begin{center}
\begin{tabular}{lcl}
Todor D. Todorov & & \underline{\hspace{12em}} \\
Adviser & \hspace{6em} & Signature\\
\vspace{0.25in}\\
Dylan Retsek & & \underline{\hspace{12em}} \\
Committee Member & \hspace{6em} & Signature\\
\vspace{0.25in}\\
Mark Stankus & & \underline{\hspace{12em}} \\
Committee Member & \hspace{6em} & Signature\\
\vspace{0.25in}\\
\end{tabular}
\end{center}
\newpage


\begin{center}
\textbf{Abstract}\\
\textbf{Reduction of the Number of Quantifiers in\\Real Analysis through Infinitesimals}\\
Raymond Cavalcante
\end{center}

\begin{quote}
We construct the non-standard complex (and real) numbers using the ultrapower method in the spirit of Cauchy's construction of the real numbers. We show that the non-standard complex numbers are a non-archimedean, algebraically closed field, and that the non-standard real numbers are a totally ordered, real-closed, non-archimedean field. We explore the various types of non-standard numbers, and develop the non-standard completeness results (Saturation Principle, Supremum Completeness of Bounded Internal Sets, etc) for $\starr$. We give non-standard characterizations for such usual topological objects as open, closed, bounded, and compact sets in terms of monads. We also consider such traditional topics of real analysis as limits, continuity, uniform continuity, convergence, uniform convergence, etc. in a non-standard setting. In both topology and real analysis we reduce (and in some cases eliminate) the number of quantifiers in the non-standard setting.
\end{quote}



\newpage


\tableofcontents
\setcounter{tocdepth}{3}

\newpage

\pagenumbering{arabic}

\chapter{Introduction}

\section{Historical Discussion}
	
	When calculus was discovered by Leibniz in the late seventeenth century, he invented the notation $dx$ to mean an infinitely small change in the variable $x$. As to the existence of an infinitely small quantity, Leibniz explained:

	\begin{quote}
	It will be sufficient if, when we speak of infinitely great (or more strictly unlimited), or of infinitely small quantities (i.e. the very least of those within our knowledge), it is understood that we mean quantities that are indefinitely great or indefinitely small, i.e., as great as you please, or as small as you please, so that the error that one may assign may be less than a certain assigned quantity... (Goldblatt \cite{rGoldblatt} p.6)
	\end{quote}
	
	When Euler developed infinite series for logarithmic, exponential, and trigonometric functions, he frequently used the ideas of arbitrarily small and arbitrarily large (Goldblatt \cite{rGoldblatt} p.8).
	
	Though the idea of infinitely large and small numbers were used to great success by Leibniz, Newton, and Euler, the larger mathematical community was skeptical of Leibniz's explanation given above. In the late nineteenth century the work of Dedekind, Cantor, Cauchy, Balzano, and Weierstrass expunged infinitesimals from analysis and replaced them with the $\varep$-$\delta$ formulations which are widely taught today.
	
	In 1966, Abraham Robinson discovered the non-standard analysis \cite{aRob66}, providing a firm foundation for the infinitesimals which were banished in the late nineteenth century. The version of non-standard analysis developed by Robinson relies heavily on formal logic. The ultrapower formulation (constructionist approach) was discovered soon after by Luxemburg \cite{wLuxNotes} (See also Stroyan and Luxemburg \cite{StroLux76}). For a more contemporary presentation of the ultrapower method we refer to Lindst\o m \cite{tLin}.
	
	Whereas the ultrapower method in Luxemburg \cite{wLuxNotes}, Lindstr\o m \cite{tLin}, and Goldblatt \cite{rGoldblatt} utilizes the natural numbers $\mathn$ as the index set, we shall use $\rplus$ as the index set. Note that another ultrapower non-standard model (with different index set and different ultrafilter) was used in Guy Berger's thesis \cite{gBerger05} for studying delta-like solutions of Hopf's equation.
	
	It is our belief that the use of infinitesimals in analysis simplifies the presentation of key concepts such as limits (of usual functions and sequences), continuity, derivatives, and topological concepts such as open, closed, and compact sets in the usual topology on $\mathr$. Furthermore, we stress the use of the non-commutative quantifiers $\forall$ and $\exists$ in standard analysis creates statements for foundational concepts which can be confusing to the beginner. We demonstrate that in the framework of non-standard analysis we are able to take the previous list of concepts and formulate them with less quantifiers (and sometimes none).
	
	For additional reading we refer to Lindstr\o m \cite{tLin} and Davis \cite{mDavis}. For the reader interested in teaching calculus in the language of infinitesimals we refer to Keisler \cite{jKeisE} \cite{jKeisF} and to Todorov \cite{tdTod2000a}.
	
\section{Summary}

	We take a constructive approach to the non-standard complex (and real) numbers. Our approach is similar in nature to the Cauchy construction of the real numbers as equivalence classes of fundamental sequences of rational numbers. Chapter \ref{C: Preliminaries} is an introduction to Non-Standard Analysis via the ultrapower construction. Section \ref{S: Filters and Ultrafilters} discusses the basic theory of filters and ultrafilters with the intention of specifying a specific ultrafilter with which to define an equivalence relation among sequences of complex numbers. In Section \ref{S: A Non-Standard Extension of C} we demonstrate that the ultrapower set modulo the equivalence relation is a field, which we call $\starc$. In Sections \ref{S: Algebraic Properties} and \ref{S: Order} we show that $\starc$ is an algebraically closed field, with $\starr$ as a real-closed, totally ordered subfield.

	In Section \ref{S: Trichotomy} we define and give examples of infinitesimal, finite, and unlimited elements of $\starc$; thus demonstrating that both $\starc$ and $\starr$ are non-archimedean fields. We introduce the standard part mapping in Section \ref{S: SPM} to connect the non-standard $\starc$ to the standard $\mathc$. It may then seem natural that the standard part mapping is the ring homomorphism which proves that the set of finite non-standard numbers modulo the infinitesimals is isomorphic to the complex numbers.
	
	In Section \ref{S: NSE of Set} we generalize the method of Section \ref{S: A Non-Standard Extension of C} for any set. In a similar spirit of generalizing our methods we develop internal sets in Section \ref{S: Internal Sets} as nets of subsets of $\rplus$, and we demonstrate that the work of Section \ref{S: NSE of Set} is a special case of the internal sets. Having defined internal sets, we then develop the completeness results for $\starr$ in Section \ref{S: Completeness}. We first prove Dedekind Completeness, which is the non-standard analogue of the supremum completeness of $\mathr$. We discuss the Spilling Principles to address the infinitesimal/finite and finite/unlimited barrier. The Saturation Principle is proven in simplified form, and the Cantor Principle is given as a direct corollary.
	
	Sections \ref{S: Transfer by Example} and \ref{S: Logic for Transfer} develop the necessary background for the Transfer Principle; these are the only sections in which we use formal logic. The Transfer Principle is a powerful theorem allowing us to literally transfer properly formed statements about $\mathr$ into statements about $\starr$, and vice versa. The Transfer Principle also tells us where to place the asterisks in the formalization of traditional real analytic statements.
	
	Chapter \ref{C: Topology} reviews open, closed, bounded, and compact sets in the context of the usual topology on $\mathr$. We introduce the monad, and think of it as a universally open infinitesimal interval. We proceed to develop the non-standard characterizations of the items in the list above. We emphasize that the non-standard characterizations reduce the number of quantifiers, and in some cases, eliminates them completely (hence the title of the work: "Reduction of the number of quantifiers..."). For example, the standard definition of a compact set has two quantifiers (For every cover there is a finite subcover). The non-standard characterization of compactness is given in terms of monads and is free of quantifiers.
	
	Chapter \ref{C: Analysis} is concerned with usual topics from real analysis in a non-standard setting. We review the standard definitions of limits, continuity, uniform continuity, sequences, uniform convergence, and derivatives. We then present and prove the non-standard characterizations, and again emphasize the reduction of quantifiers. For example, the standard definition of a limit has three quantifiers, whereas the non-standard characterization has only one quantifier. As another, more subtle, example, the standard definition of uniform continuity has four quantifiers, whereas the non-standard characterization has two quantifiers. Moreover, the non-standard setting gives a definition which conforms to what our intuition would expect uniform continuity to mean.
	
	The examples cited above are only a few instances in which we may reduce the number of quantifiers.

\chapter{Introduction to Non-Standard Analysis}\label{C: Preliminaries}

We present aspects of the theory of filters and ultrafilters in order to construct the non-standard complex (and real) numbers via the ultrapower method. This process is then generalized to turn standard sets and functions into non-standard sets and functions. It turns out that our extension of any standard set is only a special case of an internal set. Knowing what an internal set is, allows us to discuss the various completeness theorems on $\starr$, some of which may come as a surprise. We end the chapter with the Transfer Principle, which is the most general statement for how to move from a standard statement, to a non-standard statement, and vice versa.

\section{Filters and Ultrafilters}\label{S: Filters and Ultrafilters}

Let us begin with the basic theory of filters and ultrafilters. In the following, $I$ is an arbitrary infinite set and $\mathcal{P}(I)$ is the power set of $I$.

	\begin{definition}[Filter and Free Filter]\label{D: Filter and Free Filter}\index{Filter}\index{Free Filter}
		A non-empty set $\scf \subeq \poweri$ is a \textbf{filter} on $I$ if for each $A, B \in \poweri$ we have
	\begin{quote}
		\begin{description}
		
			\item[(a)] $\varnothing \notin \scf$.
			\item[(b)] $A, B \in \scf$ implies $A \cap B \in \scf$.
			\item[(c)] $A \in \scf$ and $A \subeq B \subeq I$ implies $B \in \scf$.
		
		\end{description}
	Further, a filter $\scf$ is a \textbf{free filter} if:
		\begin{description}
			
			\item[(d)] $\cap_{A \in \scf} A = \varnothing$
		
		\end{description}
	A filter $\scf$ is called \textbf{countably incomplete} if:
		\begin{description}
		
			\item[(e)] There is a decreasing sequence of sets $I = I_0 \supset I_1 \supset I_2 \supset ...$ in $\scf$ such that $\bigcap_{n=0}^{\infty} I_n = \varnothing$.
		
		\end{description}
	\end{quote}
	\end{definition}

	\begin{definition}[Ultrafilter]\label{D: Ultrafilter}\index{Ultrafilter}
		A filter $\scu$ on a set $I$ is an \textbf{ultrafilter} (or maximal filter) if there is no filter $\scf$ on $I$ that properly contains $\scu$.
	\end{definition}

	\begin{theorem}[Characterization of Ultrafilters]\label{T: Characterization of Ultrafilter}
		Let $\scu$ be a filter on $I$, then $\scu$ is an ultrafilter on $I$ if and only if for every $A \subeq I$, either $A \in \scu$ or $I \setminus A \in \scu$.
	\end{theorem}
	\begin{proof} 
		\begin{description}
		
		\item[($\Rightarrow$)] Let $\scu$ be an ultrafilter on $I$, and suppose to the contrary that for a fixed $A \subset I$, neither $A \notin \scu$ nor $I \setminus A \notin \scu$. Define $\overline{\scu} =: \{ X \in \mathcal{P}(I) : A \cup X \in \scu\}$. It is painless to show that $\overline{\scu}$ is a filter on $I$. Note, $\scu \subeq \overline{\scu}$ by taking $B \in \scu$ and using \textbf{(c)} from Definition \ref{D: Filter and Free Filter}. Also, $I \setminus A \in \overline{\scu}$ since $A \cup (I \setminus A) = I \in \scu$, and by assumption $I \setminus A \notin \scu$. So $\scu$ a proper subset of $\overline{\scu}$. Which contradicts the assumption that $\scu$ is a maximal filter.
		\item[($\Leftarrow$)] Let $\scu$ be a filter on $I$, and assume for every $A \subeq I$, either $A \in \scu$ or $I \setminus A \in \scu$, and assume $\scu$ is not an ultrafilter. So there exists $\scv$, a filter on $I$, such that $\scu$ is properly contained by $\scv$. Let $A \in \scv \setminus \scu$. Then $A \cap X \not= \varnothing$ for all $X \in \scu \subset \scv$ (by \textbf{(b)} of Definition \ref{D: Filter and Free Filter}). Particularly, when $X = I \setminus A$ we have $A \cap (I \setminus A) \not= \varnothing$, which is a contradiction.
		
		\end{description}
	\end{proof}

	\begin{corollary}[A Generalization]\label{C: Extension of Characterization}
		Let $\scu$ be an ultrafilter on $I$. If $\bigcup_{n=1}^{m} A_n \in \scu$ for some $m \in \mathn$, then $A_n \in \scu$ for \emph{at least} one $n$. Additionally, if the sets $A_n$ ($n = 1, 2, ..., m$) are mutually disjoint, then $A_n \in \scu$ holds for \emph{exactly} one $n$.
	\end{corollary}
	\begin{proof}
		We take Theorem \ref{T: Characterization of Ultrafilter} as the base case and apply induction.
	\end{proof}

	\begin{examples}$ $
	\begin{quote}
		\begin{description}
		
			\item[(i)] Let $I = \rplus$ and let $i \in \rplus$. The set $\scf_0 =: \{ A \in \mathcal{P}(\rplus) : i \in A \}$ is a filter which is not free. Indeed, $\varnothing \not\in \scf_0$ since $\rplus \not= \varnothing$. Suppose $A, B \in \scf_0$ so that $i \in A$ and $i \in B$. Then $i \in A \cap B$, therefore $A \cap B \in \scf_0$. Now, let $A \in \scf_0$ and suppose $A \subseteq B \subseteq \rplus$. Then $i \in A$ and by assumption $i \in B$, therefore $B \in \scf_0$. Hence, $\scf_0$ is a filter. To show that $\scf_0$ is not free note that $i \in A$ for all $A \in \scf_0$ and so $\{ i \} \subseteq \cap_{A \in \scf_0} A \not= \varnothing$.
			\item[(ii)] Let $I = \mathn$. The set $\scf_{r} =: \{ A \in \mathcal{P}(\mathn) : (\mathn \setminus A)$ is finite $\}$ is the Frechet filter which is not an ultrafilter. Indeed, $\varnothing \not\in \scf_{r}$ for if it were then $\mathn \setminus \varnothing = \mathn$ would be finite, which is a contradiction. Let $A, B \in \scf_{r}$. Then $(\mathn \setminus A)$ and $(\mathn \setminus B)$ are finite. Therefore, $(\mathn \setminus A) \cup (\mathn \setminus B)$ is finite. By DeMorgan's Laws we have $\mathn \setminus (A \cap B)$ is finite. Hence, $A \cap B \in \scf_{r}$. Next, let $A \in \scf_{r}$ and $A \subseteq B \subseteq \mathn$. Then $(\mathn \setminus A)$ is finite and since $A \subseteq B$ we have $(\mathn \setminus B) \subseteq (\mathn \setminus A)$ giving $(\mathn \setminus B)$ is finite. Hence, $B \in \scf_{r}$. To see that $\scf_{r}$ is not an ultrafilter consider the set of even numbers in $\mathn$, which are infinite, and therefore not in $\scf_r$. Notice that the complement of the evens is the odds, which are also infinite, and therefore not in $\scf_r$. So we have found a set for which neither it, nor its compliment is in the filter. Therefore, $\scf_r$ is not an ultrafilter by Theorem \ref{T: Characterization of Ultrafilter}.
		
		\end{description}
	\end{quote}
	\end{examples}

	We now focus our attention on the existence of the ultrafilter. To that end, let $\scs$ be a partially ordered set with $\leq$ a partial order. An element $m \in \scs$ is a \textbf{maximal element} of $\scs$ if there is no $s \in \scs$ such that $m \lneqq s.$

	\begin{lemma}[Zorn's Lemma]\label{L: Zorn's Lemma}\index{Zorn's Lemma}
		If $\scs$ is a partially ordered set such that every totally ordered subset $\scl$ of $\scs$ has an upper bound $b \in \scs$, then $\scs$ has maximal elements.
	\end{lemma}

	Recall that Zorn's Lemma is equivalent to the Axiom of Choice, and thus can not be proven.
	
	\begin{theorem}[Existence of Ultrafilter on $I$]\label{T: Existence}
		Let $I$ be an infinite set and $I_n \subseteq I$ such that $$I = I_0 \supset I_1 \supset I_2 \supset ... \textnormal{ and } \bigcap^{\infty}_{n=0} I_n = \varnothing.$$ Then there exists a free ultrafilter $\mathcal{U}$ on $I$ such that $I_n \in \mathcal{U}$ for all $n \in \mathn$.
	\end{theorem}
	\begin{proof} 
		Let $\mathcal{S}$ be the set of all free filters $\mathcal{F}$ on $I$ such that $I_n \in \mathcal{F}$ for all $n \in \mathn$. We order $\mathcal{S}$ by set inclusion, $\subseteq$. Note, $\mathcal{S} \not= \varnothing$ since $\mathcal{F}_0 = \{ A \in \mathcal{P}(I) : I_n \subseteq A$ for some $n \}$ is a free filter on $I$ and $I_n \in \mathcal{F}_0$ for all $n$. Let $\mathcal{L} \subseteq \mathcal{S}$ be a totally ordered subset of $\mathcal{S}$. So each $L \in \mathcal{L}$ is a free filter on $I$ with $I_n \in L$ for all $n$. 
		
		Define $\mathcal{L}^* = \cup_{L \in \mathcal{L}} L$. Then $\mathcal{L}^*$ is also a free filter on $I$. Indeed, $\varnothing \notin \mathcal{L}^*$ since $\varnothing \in \mathcal{L}$ would imply $\varnothing \in L$ for some $L \in \mathcal{L}$, which is a contradiction. Next, let $A, B \in \mathcal{L}^*$. Then $A \in L_1$ and $B \in L_2$ for some $L_1, L_2 \in \mathcal{L}$. Without loss of generality, let $L_1 \subseteq L_2$ so $A \in L_2$ and $B \in L_2$. Therefore $A \cap B \in L_2$ so $A \cap B \in \mathcal{L}^*$. Lastly, let $A \in \mathcal{L}^*$ and $A \subseteq B \subseteq I$. Since $A \in \mathcal{L}^*$, we have $A \in L_1$ for some $L_1 \in \mathcal{L}$ hence $L_1$ a free filter implies $B \in L_1$ so $B \in \mathcal{L}^*$. Therefore $\mathcal{L}^* \in \mathcal{S}$, and $\mathcal{L}^*$ is an upper bound for $\mathcal{L}$, so Zorn's Lemma implies $\mathcal{S}$ has a maximal element. Let $\mathcal{U} \in \mathcal{S}$ be such a maximal element, then $\mathcal{U}$ is a filter on $I$ and $I_n \in \mathcal{U}$ for all $n$.
	\end{proof}

\section{Probability Measure on $\powerplus$}\label{S: Probability Measure}

	We now fix $I = \rplus$, $I_n = (0, \frac{1}{n})$, and $\scu$ to be an ultrafilter on $\rplus$ with $I_n \in \scu$ for all $n \in \mathn$. We know such an ultrafilter exists by Theorem \ref{T: Existence}. The purpose of this excursion into probability measures is to justify our later use of the phrase "almost everywhere", which will ease our notational burden in some instances.

	\begin{definition}[Probability Measure]\label{D: Probability Measure}\index{Probability Measure}
		Let $p: \powerplus \to \{0,1\}$ be a two-valued \textbf{probability measure} with the following properties:
	\begin{quote}
		\begin{description}
		
			\item[(a)] The probability measure $p$ is finitely additive. That is, $p(A \cup B) = p(A) + p(B)$ for disjoint $A$ and $B$,
			\item[(b)] We have $p(I_n) = 1$, where $I_n = \{ \varep \in \rplus : 0 < \varep < \frac{1}{n}\}$, and
			\item[(c)] For every finite subset $A$ of $\rplus$, $p(A) = 0$; in particular, $p(\varnothing) = 0$.
		
		\end{description}
	\end{quote}
	\end{definition}

	\begin{theorem}[Ultrafilter implies Probability Measure] \label{T: Ultrafilter implies Probability Measure}
		Let $\scu$ be an ultrafilter. Define the probability measure $p: \powerplus \to \{0,1\}$ by
		\begin{align}
			p(A) =	
				\begin{cases}
					0 &\text{if $A \notin \scu$}\\ 
					1 &\text{if  $A \in \scu$}\notag
				\end{cases}
		\end{align}
		Then $p$ is a finitely additive measure which satisfies \emph{\textbf{(a)} - \textbf{(c)}} of Definition \ref{D: Probability Measure}.
	\end{theorem}

	The above theorem in addition to the converse, which implies the existence of an ultrafilter given a probability measure as defined in Definition \ref{D: Probability Measure}, establishes a one-to-one correspondence between our fixed ultrafilter $\scu$ and our finitely additive two-valued probability measure $p$. Given the existence of $\scu$ on $\rplus$ from Theorem \ref{T: Existence}, we are guaranteed the existence of such a probability measure.
	
	We now develop the meaning of "almost everywhere" in the language of predicates.

	\begin{definition}[Almost Everywhere]\label{D: Predicate}\index{Almost Everywhere}
		Let $P: \rplus \to \{ \textnormal{true, false}\}$ be a predicate in one variable on $\rplus$. We say that $P$ holds \textbf{almost everywhere} (written a.e.) if $p (\{ \varep \in \rplus : P(\varep) \textnormal{ is true} \} ) = 1$. Alternately $P$ holds a.e. if $\{ \varep \in \rplus : P(\varep) \textnormal{ is true} \}\in \scu$.
	\end{definition}
	
	Before presenting examples of predicates, we introduce a definition. For non-empty sets $A$ and $B$, the set of all functions from $A$ to $B$ is the \textbf{ultrapower set}, which is denoted by $B^A$. Let $S$ be a set. We call the elements of $S^{\rplus}$ \textbf{nets} in $S$ and we denot them $(\aep) \in S^{\rplus}$, where $\varep \in \rplus$.

	\begin{examples}
		The following are examples of predicates:
	\begin{quote}
			\begin{description}
			
				\item[(i)] Let $(\aep), (\bep) \in \crplus$. Let $P(\varep)$ be the predicate "$\aep = \bep$".
				\item[(ii)] Let $(\aep) \in \crplus$ and $\mathcal{S} \subseteq \mathc$. Let $P(\varep)$ be the predicate "$\aep \in \mathcal{S}$".
				\item[(iii)] Let $(\aep) \in \crplus$ and $(\Sep) \in \mathcal{P}(\mathc)^{\rplus}$. Let $P(\varep)$ be the predicate "$\aep \in \Sep$".
				\item[(iv)] Let $(\aep), (\bep) \in \mathr^{\rplus}$. Let $P(\varep)$ be the predicate "$\aep < \bep$".
			
			\end{description}
	\end{quote}
	\end{examples}

\section{A Non-Standard Extension of $\mathc$}\label{S: A Non-Standard Extension of C}
	
	We now extend the standard complex numbers $\mathc$ to the non-standard complex numbers $\starc$. This will be done by means of ultrapowers and the ultrafilter which was fixed in Section \ref{S: Probability Measure}.
	
	It is clear that $\crplus$ is a commutative ring with unity and the set $J = \{ (\aep) \in \crplus : \aep = 0 \textnormal{ a.e.} \}$ is an ideal of $\crplus$. Recall from Definition \ref{D: Predicate}, this means $\{ \varep \in \rplus :\aep = 0\} \in \scu$.
	
	\begin{definition}[Equivalence Relation]\label{D: Equivalence Relation}
		For $(\aep), (\bep) \in \crplus$, we say $(\aep) \sim (\bep)$ \iff $\aep = \bep$ a.e. Alternately (from Definition \ref{D: Predicate}), we say $(\aep) \sim (\bep)$ \iff $\{ \varep \in \rplus : \aep = \bep \} \in \scu$.
	\end{definition}

	\begin{definition}[Equivalence Classes]\label{D: Equivalence Classes}
		If $(\aep) \in \crplus$, we denote by $\bra \aep \ket$ the equivalence class of $(\aep)$. We denote the corresponding factor ring $\crplus / \sim$ by $\starc$.
	\end{definition}
	
	At the very least we would like $\starc$ as defined above to be a field. Indeed, we will show it is. However, we will get so much more. Eventually, we show $\starc$ is a non-archimedian, algebraicially closed field which has $\mathc$ as an embedded subset.

	\begin{theorem}\label{T: Field}
		If $J = \{ (\aep) \in \crplus : \aep = 0$ a.e.$\}$, as before, then $J$ is a maximal ideal of $\crplus$; therefore $\starc$ is a field.
	\end{theorem}
	\begin{proof} 
		Suppose $K \subeq \crplus$ is an ideal such that $J \subsetneqq K \subeq \crplus$. Let $(\aep) \in K \setminus J$. That is, $A =: \{ \varep \in \rplus : \aep \not= 0 \} \in \scu$. Define $(\bep) \in \crplus$ by
		\begin{align}
			\bep =	
				\begin{cases}
					\frac{1}{\aep} &\text{ if $\varep \in A$}\\ 
					1 &\text{if $\varep \in \rplus \setminus A$}\notag
				\end{cases}
		\end{align}
		Let $B = \{ \varep \in \rplus : \aep \cdot \bep = 1\}$. We have $A \subeq B$ since $\varep \in A$ implies $\bep = \frac{1}{\aep}$ and $\aep \cdot \bep = 1$. By \textbf{(c)} of Definition \ref{D: Filter and Free Filter} we have $B \in \scu$, so $\aep \cdot \bep = 1$ a.e. Hence $(\aep)$ has multiplicative inverse $(\bep)$ as defined above, and $(1) \in K$. Therefore $\crplus \subeq K$ and $K = \crplus$. So $J$ is a maximal ideal and $\starc$ is a field.
	\end{proof}

	By embedding $\mathc$ in $\starc$ we may discuss standard elements of $\mathc$ as elements of $\starc$ so as to compare standard to non-standard elements.

	\begin{definition}\label{D: Embedding}\index{Embedding}
		We \textbf{embed} $\mathc$ in $\starc$, denoted $\mathc \hookrightarrow \starc$, by $c \to \bra \cep \ket$, with $\cep = c$ for all $\varep \in \rplus$. We write $\bra c \ket$ instead of $\bra \cep \ket$, and $\mathc \subset \starc$ rather than the more precise $\mathc \hookrightarrow \starc$.
	\end{definition}

	\begin{lemma}\label{L: Proper Extension}
		The non-standard extension $\starc$ is a proper extension of $\mathc$. That is, $\starc \setminus \mathc \not= \varnothing$.
	\end{lemma}
	\begin{proof}
		Let $\aep = \varep$. We demonstrate that $\bra \varep \ket \in \starc \setminus \mathc$. Assume to the contrary that $\bra \varep \ket = \bra c \ket$ for some $c \in \mathc$. Then $A = \{ \varep \in \rplus : \varep = c\} \in \scu$. Observe that $A \not= \varnothing$ so $c \in \rplus$ implying $A = \{ c \} \in \scu$, a contradiction since $p(\{c\}) = 0$. Therefore $\bra \varep \ket \not= \bra c \ket$ for any $c \in \mathc$. So $\bra \aep \ket = \bra \varep \ket \notin \mathc$.
	\end{proof}

\section{Algebraic Properties of $\starc$ and $\starr$}\label{S: Algebraic Properties}

	We demonstrate that $\starc$ is an algebraically closed field and that $\starr$ is a real-closed field.

	\begin{theorem}\label{T: Algebraically Closed}
		The non-standard extension $\starc$ is an algebraically closed field. That is, every polynomial of non-zero degree with coefficients in $\starc$ has a root in $\starc$.
	\end{theorem}
	\begin{proof} 
		Let $P \in \polycstar$ with $\deg P \geq 1$. In other words, let $P(x) = \sum_{n=0}^{m} \alpha_n x^n$, where $\alpha_n \in \starc$ and $\alpha_m \not= 0$. Since $\alpha_n \in \starc$, $\alpha_n = \bra a_{n\varep} \ket$ for some $(a_{n\varep}) \in \crplus$ ($n = 0, 1, ... , m$). Define $P_\varep(x) := \sum_{n=0}^{m} a_{n\varep} x^n$. For any fixed $\varep \in \rplus$, $P_\varep \in \mathc[x]$. Since $\mathc$ is algebraically closed there exists $z_\varep \in \mathc$ such that $P_\varep(z_\varep) = 0$ for all $\varep \in \rplus$.

		Now, define $z = \bra z_\varep \ket \in \starc$. We have $\{ \varep \in \rplus : P_\varep(z_\varep) = 0 \} = \rplus \in \scu$. Define the internal function $\bra P_\varep \ket \in \polycstar$ where $\bra P_\varep \ket : \starc \to \starc$ by $\bra P_\varep \ket ( \bra c_\varep \ket ) = \bra P_\varep (c_\varep) \ket$. So $\bra P_\varep \ket ( \bra z_\varep \ket ) = \bra P_\varep ( z_\varep ) \ket = \bra 0 \ket = 0$

		We must now show that $\bra P_\varep \ket = P$. Let $x = \bra x_\varep \ket \in \starc$ (arbitrarily). Then $\bra P_\varep \ket(x) = \bra P_\varep \ket (\bra x_\varep \ket) \overset{\text{def}}{=} \bra P_\varep (x_\varep) \ket = \bra \sum_{n=0}^{m} a_{n\varep} x_{\varep}^{n} \ket = \sum_{n=0}^{m} \bra a_{n\varep} \ket \bra x_{\varep}^{n} \ket = \sum_{n=0}^{m} \bra a_{n\varep} \ket \bra x_\varep \ket^{n} = \sum_{n=0}^{m} \alpha_n \bra x_\varep \ket^n = \sum_{n=0}^{m} \alpha_n x^n = P(x)$.
	\end{proof}

	We now take the short route to showing that $\mathr$ is a real closed field.
	
	\begin{definition}\label{D: Real Closed}
		An arbitrary field $\mathbb{K}$ is \textbf{real closed} if
	\begin{quote}
			\begin{description}
			
				\item[(a)] For all $a \in \mathbb{K}$ either $x^2 = a$ or $x^2 = -a$ is solvable in $\mathbb{K}$.
				\item[(b)] Every polynomial of odd degree with coefficients in $\mathbb{K}$ has a root in $\mathbb{K}$.
				
			\end{description}
	\end{quote}
	\end{definition}

	Recall that each real closed field $\mathbb{K}$ is uniquely orderable by $x \geq 0$ \iff there exists $y\in \mathbb{K}$ such that $x = y^2$ (Van Der Waerden \cite{VanDerWaerden}, ch. 11).
	
	\begin{theorem}[Artin-Schrier]
		If $\mathbb{K}$ is a totally ordered field, then the following are equivalent:
	\begin{quote}
			\begin{description}
			
				\item[(a)] $\mathbb{K}$ is a real-closed field.
				\item[(b)] $\mathbb{K}(i)$ is an algebraically closed field.
			
			\end{description}
	\end{quote}
	\end{theorem}
	
	We refer, again, to \cite{VanDerWaerden}.
	
	\begin{lemma}\label{L: Adjoin i}
		The non-standard extension $\starc = \starr(i)$ in the sense that each $z \in \starc$ can be uniquely represented as $z = a + bi$ where $a, b \in \starr$.
	\end{lemma}
	\begin{proof}
		Let $z = \bra z_\varep \ket$ for some $(z_\varep) \in \crplus$. For each fixed $\varep \in \rplus$, $\zep = \aep + \bep i$. Then $a = \bra \aep \ket$ and $b = \bra \bep \ket$ where $\aep, \bep \in \mathr$ for all $\varep$.
	\end{proof}
		
	\begin{corollary}\label{C: Real Closed}
		The non-standard extension $\starr$ is a real-closed field.
	\end{corollary}
	\begin{proof}
		From Theorem \ref{T: Algebraically Closed}, Lemma \ref{L: Adjoin i}, and in light of the Theorem of Artin-Schrier \cite{VanDerWaerden}, $\starr$ is a real-closed field.
	\end{proof}

\section{Order Relation in $\starr$}\label{S: Order}
	The question may arise as to whether the richness of $\starr$ destroys the total ordering that we hope to inherit from $\mathr$. This turns out not to be the case.

	\begin{definition}
		Let $\bra \aep \ket, \bra \bep \ket \in \starr$. Then $\bra \aep \ket < \bra \bep \ket$ if $\aep < \bep$ a.e. Alternatively, $\bra \aep \ket < \bra \bep \ket$ if $\{ \varep : \aep < \bep \} \in \scu$.
	\end{definition}

	\begin{theorem}[Totally Ordered Field]\label{T: Totally Ordered Field}$ $
	\begin{quote}
		\begin{description}
			
			\item[(a)] The non-standard real numbers $\starr$ are a totally ordered field.
			\item[(b)] The real numbers $\mathr$ are a totally ordered subfield of $\starr$.
			
		\end{description}
	\end{quote}
	\end{theorem}
	\begin{proof}
		We leave reader to check that $\leq$ is an order relation on $\starr$ (i.e. reflexive, anti-symmetric, and transitive).
	\begin{description}	
		
		\item[(a)] Let $\bra \aep \ket, \bra \bep \ket \in \starr$. We must show that $\bra \aep \ket < \bra \bep \ket$, $\bra \aep \ket = \bra \bep \ket$, or $\bra \aep \ket > \bra \bep \ket$. Indeed, $\{ \varep \in \rplus : \aep < \bep \} \cup \{ \varep \in \rplus : \aep = \bep \} \cup \{ \varep \in \rplus : \aep > \bep \} = \rplus$ since $\mathr$ is totally ordered. By Corollary \ref{C: Extension of Characterization}, only one of the above disjoint sets may be in the ultrafilter $\scu$, as desired.
		
		\item[(b)] We have to show that $0 \notin \starr_+$ (which is evidently true), and that $\starr_+$ is closed under addition and multiplication. We shall leave addition to the reader and prove closure under multiplication. Let $\bra \aep \ket, \bra \bep \ket \in \starr_+$ with $\bra \aep \ket > 0$ and $\bra \bep \ket > 0$. Then $A =: \{ \varep \in \rplus : \aep > 0\} \in \scu$ and $B =: \{ \varep \in \rplus : \bep > 0\} \in \scu$. Considering $C =: \{ \varep \in \rplus : \aep\bep > 0\}$, let $\varep \in A \cap B$, then $\aep > 0$ and $\bep > 0$ so $\aep\bep > 0$ (in $\mathr$) implying $\varep \in C$. Therefore, $A \cap B \subseteq C$ and by \textbf{(c)} of Definition \ref{D: Filter and Free Filter}, $C \in \scu$.
		
	\end{description}
	\end{proof}

\section{Non-Standard Numbers}\label{S: Trichotomy}
	We classify the non-standard complex numbers into two distinct sets. We then explore the properties of and among these sets. Looking forward, we view $\mathbb{Q} \subset \starc$ (as an embedding in the spirit of Definition \ref{D: Embedding}) so that we may apply the order properties from Section \ref{S: Order}.

	\begin{definition}[Classification]\label{D: Trichotomy}\index{Finite Non-Standard Number}\index{Infinitesimal}\index{Infinitely Large Non-Standard Number}
		Let $z \in \starc$.
	\begin{quote}
		\begin{description}
			
			\item[(a)] If there exists an $n \in \mathn$ such that $|z| \leq n$ then $z$ is \textbf{finite}. We denote by $\scf(\starc)$ the set of finite numbers.
			\item[(b)] If $|z| < \frac{1}{n}$ for all $n \in \mathn$, then $z$ is \textbf{infinitesimal}. We denote by $\sci(\starc)$ the set of infinitesimal numbers.
			\item[(c)] If $n < |z|$ for all $n \in \mathn$, then $z$ is \textbf{infinitely large}. We denote by $\scl(\starc)$ the set of infinitely large numbers.
			
		\end{description}
	\end{quote}
	\end{definition}
	
	It is important to note that $\scf(\starc)$ and $\scl(\starc)$ are disjoint, as are $\scl(\starc)$ and $\mathc$. Furthermore, the finite numbers combined with the infinitely large numbers give all of $\starc$. That is, $\scf(\starc) \cup \scl(\starc) = \starc$. We also have $\mathc \cap \sci(\starc) = \{0\}$ and $\mathr \cap \sci(\starr) = \{ 0\}$.

	\begin{lemma}\label{L: Finite Integral Domain}
		The set of finite non-standard numbers $\scf(\starc)$ is an integral domain.
	\end{lemma}
	\begin{proof}
		It is a straightforward proof by definition to show that $\scfstarc$ is a subring of $\starc$. It may be questionable to the reader as to whether $\scfstarc$ has zero divisors in $\starc$ since $\crplus$ is a ring with zero divisors and $\starc$ is simply $\crplus \setminus \sim$ (as in Section \ref{S: A Non-Standard Extension of C}).
		
		Indeed, since $x, y \in \scf(\starc) \subseteq \starc$ if $xy=0$ then either $x=0$ or $y=0$ since $\starc$ is a field. Hence, $\scf(\starc)$ is an integral domain. 
	\end{proof}
	
	\begin{lemma}\label{L: Maximal Convex}
		The set of infinitesimals $\sci(\starc)$ is a convex maximal ideal in $\scf(\starc)$ in the sense that $\sci(\starc)$ is a maximal ideal in $\scf(\starc)$ such that if $|x| \leq |y| \in \sci(\starc)$, then $x \in \sci(\starc)$.
	\end{lemma}
	\begin{proof}
		First we show that $\sci(\starc)$ is an ideal of $\scf(\starc)$. Clearly, $\scistarc \subseteq \scfstarc$. Let $x, y \in \scistarc$. Then for all $n \in \mathn$ $|x| < \frac{1}{n}$ and $|y| < \frac{1}{n}$. Then $|x + y| \leq |x| + |y| < \frac{2}{n}$ for all $n \in \mathn$. Therefore, $x + y \in \scistarc$. Next, let $x \in \scfstarc$, $y \in \scistarc$. Then there exists $n_1 \in \mathn$ such that $|x| \leq n_1$ and $|y| < \frac{1}{n}$ for all $n \in \mathn$. Hence, $|xy| < \frac{n_1}{n}$ for all $n \in \mathn$. Therefore, $xy \in \scistarc$. So $\scistarc$ an ideal of $\scfstarc$.
		
		Next we show that $\scistarc$ is maximal in $\scfstarc$. Suppose $K \subseteq \scfstarc$ is an ideal so that $\scistarc \subsetneqq K \subseteq \scfstarc$. So there exists $x \in K \setminus \sci(\starc)$. That is, there exists $n_1 \in \mathn$ such that $|x| \geq \frac{1}{n_1}$. Hence, $\frac{1}{|x|} \leq n_1$ and so $\frac{1}{x} \in \scf(\starc)$, which implies $x \frac{1}{x} = 1 \in K$ since $K$ is an ideal. Therefore $K = \scfstarc$ and $\scistarc$ is a maximal ideal in $\scfstarc$.
		
		The final item to prove is the convexity of the maximal ideal $\scistarc$, and this is straightforward. Let $|y| \in \scistarc$ with $|x| \leq |y|$. Then for all $n \in \mathn$, $|y| < \frac{1}{n}$ so $|x| \leq |y| < \frac{1}{n}$ putting $x \in \scistarc$. 
	\end{proof}

	It is now possible to imagine the factor ring $\scf(\starc) / \sci(\starc)$, but we wait for the next section to discuss it in any detail.

	\begin{definition}[Infinitesimal Relation]\label{D: Infinitely Close}\index{Infinitesimal Relation}
		The symbol $\approx$ denotes the infinitesimal relation. We write $z \approx 0$ \iff $z \in \sci(\starc)$. Similarly, $a \approx b$ \iff $a - b \approx 0$.
	\end{definition}
	
	We now explore some of the properties of the sets of non-standard numbers.

	\begin{corollary}
		The infinitesimal relation $\approx$ is an equivalence relation on $\starc$ that preserves addition in $\starc$ and multiplication by a scalar in $\scf(\starc)$. That is,
	\begin{quote}
			\begin{description}
			
				\item[(a)] For all $x, y, z, t \in \starc$, if $x \approx y$ and $z \approx t$, then $x + z \approx y + t$.
				\item[(b)] For all $x, y \in \starc$ and for all $\lambda \in \scf(\starc)$, if $x \approx y$, then $\lambda x \approx \lambda y$.
			
			\end{description}
	\end{quote}
		Additionally, $\approx$ preserves all ring operations in $\scf(\starc)$. That is to say,
	\begin{quote}
			\begin{description}
			
				\item[(c)] For all $x, y, z, t \in \scf(\starc)$, if $x \approx y$ and $z \approx t$, then $xz \approx yt$.
			
			\end{description}
	\end{quote}
	\end{corollary}
	\begin{proof}
		Parts \textbf{(a)} and \textbf{(b)} stem from Definition \ref{D: Infinitely Close}; specifically that $a \approx b$ \iff $a-b\approx 0$. Part \textbf{(c)} stems from the fact that $\sci(\starc)$ is an ideal in $\scf(\starc)$.
	\end{proof}

	The proofs of both of the following lemmas are found directly from Definition \ref{D: Trichotomy}.
	
	\begin{lemma}
		If $z \in \starc$ is infinitesimal, finite, or infinitely large, then so is $|z|$, respectively.
	\end{lemma}

	\begin{lemma}\label{L: Reciprocal}
		Let $z \in \starc$ such that $z \not= 0$. Then $z \in \scl(\starc)$ \iff $\frac{1}{z} \in \sci(\starc)$.
	\end{lemma}

	\begin{example}[Canonical Infinitesimal]\label{E: Canonical Infinitesimal}\index{Canonical Infinitesimal}
		Let $\rho = \bra \varep \ket$ for $\varep \in \rplus$. Then $\rho > 0$, $\rho \approx 0$, and $\rho$ is called the canonical infinitesimal. Recall that we fixed $\scu$ in  Section \ref{S: Probability Measure}  to be an ultrafilter such that $I_n = (0, \frac{1}{n}) \in \scu$ for all $n \in \mathn$. We need $\bra 0 \ket < \bra \varep \ket < \bra \frac{1}{n} \ket$, where we think of $\bra 0 \ket$ and $\bra \frac{1}{n} \ket$ as elements of $\mathbb{Q}$ embedded in $\starc$ as in Definition \ref{D: Embedding}. Since $I_n = \{ \varep \in \rplus : 0 < \varep < \frac{1}{n}\} \in \scu$ for any $n \in \mathn$, the desired string of inequalities is satisfied, thus $\rho \in \sci(\starc)$.
	\end{example}
	
	\begin{example}[More Infinitesimals]
		Given Example \ref{E: Canonical Infinitesimal} and the fact that $\sci(\starc)$ is an ideal, we immediately know that $\rho^2, \rho^3, ..., \rho^n, ...$ are infinitesimals. Also, since $\starr$ is a real-closed field (Theorem \ref{T: Algebraically Closed}), the solutions to $x^2 = \rho$ for $x \geq 0$ exist and are unique. Therefore, $\sqrt{\rho}, \sqrt[3]{\rho}, ..., \sqrt[n]{\rho}, ...$ are infinitesimals.
	\end{example}
	
	\begin{example}[Finite, Non-Standard Number]
		It is possible to have a finite number $z \in \starc$ which is neither infinitesimal, nor standard (in the sense that it is in $\mathc$). Indeed, $5 + \rho$, where $\rho$ is the canonical infinitesimal from above, is a finite, non-standard number.
		
		A plethora of examples abound. For instance, a rational expression such as $\frac{4 + \rho^2}{3 + \rho}$ is a finite, non-standard number. Indeed, $\frac{4 + \rho^2}{3 + \rho} \approx \frac{4}{3}$.
	\end{example}
	
	\begin{example}[Canonical Infinitely Large Number]\label{E: Canonical Infinitely Large}\index{Canonical Infinitely Large Number}
		Let $(\nu_\varep) \in \mathn^{\mathr_{+}}$. Define
			\begin{align}
				\nu_\varep =	
				\begin{cases}
					n &\text{if $\varep \in I_n \setminus I_{n+1}$}\\ 
					1 &\text{if  $\varep \geq 1$}\notag
				\end{cases}
			\end{align}
		For $n = 1, 2, 3, ...$ We claim that $\bra \nu_\varep \ket$ is an infinitely large positive number. We must show that for all $m \in \mathn$, $\{ \varep \in \rplus : \nu_\varep > m \} \in \scu$. Recall that $\scu$ was fixed to be an ultrafilter such that $I_n = (0, \frac{1}{n}) \in \scu$ for all $n \in \mathn$. Let $m \in \mathn$ be arbitrary and fixed. Suppose $\epn \in I_m$. Then there exists $n \in \mathn$ for which $m < n$ and $\epn \in I_{n+1} \setminus I_n$. Therefore, $\nu_{\epn} = n > m$ and so $I_m \in \{ \varep \in \rplus : \nu_\varep > m\}$. Therefore by Definition \ref{D: Filter and Free Filter} part \textbf{(c)}, we have $\{\varep \in \rplus : \nu_\varep > m\} \in \scu$. As $m$ was arbitrary, this holds for all $m \in \mathn$. Therefore $\bra \nu_\varep \ket$ is infinitely large.
	\end{example}
	
	\begin{example}[Infinitely Large Numbers]
		Similarly, $\nu^2, \nu^3, ...$ are also infinitely large numbers in $\starn \setminus \mathn$. In fact, $\starn \setminus \mathn$ consists of infinitely large numbers only and $\scf(\starn) = \mathn$. That is to say, there are no infinitesimals in the non-standard natural numbers $\starn$.
	\end{example}
	
	\begin{example}[More Infinitely Large Numbers]
		In view of Lemma \ref{L: Reciprocal} and Example \ref{E: Canonical Infinitesimal}, the following are infinitely large numbers: $$\frac{1}{\rho}, \frac{1}{\rho^2}, ..., \frac{1}{\rho^n}, ..., \frac{1}{\sqrt{\rho}}, \frac{1}{\sqrt[3]{\rho}}, ..., \frac{1}{\sqrt[n]{\rho}}, ...$$
	\end{example}
	
	\begin{examples}In light of the examples in Section \ref{S: Functions} we have additional examples of infinitesimals and infinitely large numbers.
	\begin{quote}
		\begin{description}
		
			\item[(i)] Let $\ln \rho =: \bra \ln \varep \ket$. Then $\ln \rho$ is an infinitely large negative number.
			\item[(ii)] Let $\sin \rho =: \bra \sin \varep \ket$. Then $\sin \rho$ is a positive infinitesimal.
			\item[(iii)] Let $e^{\frac{1}{\rho}} =: \bra e^{\frac{1}{\varep}} \ket$. Then $e^{\frac{1}{\rho}}$ is an infinitely large number.
		
		\end{description}
	\end{quote}
	\end{examples}

\section{Standard-Part Mapping}\label{S: SPM}
	In Section \ref{S: Trichotomy} we used the infinitesimal relation $\approx$ to relate non-standard numbers in $\starc$ to standard numbers in $\mathc$. The Standard-Part Mapping ($\st$-mapping) is another way to express this relationship. Moreover it gives us a field isomorphism between the factor ring $\scf(\starc) / \sci(\starc)$ and $\mathc$.
	
	\begin{definition}[Standard-Part Mapping]\label{D: SPM}\index{Standard Part Mapping}
		We begin with $\scf(\starr)$ as the domain:
		\begin{quote}
			\begin{description}
			
				\item[(a)] $\st : \scf(\starr) \to \mathr$ is defined to be $\st (x) = \sup \{ r \in \mathr : r \leq x \}$. For $x \in \scf(\starr)$, we call $\st(x)$ the \textbf{standard part} of $x$.
			
			\end{description}
		\end{quote}
		We may extend the above to $\scf(\starc)$:
		\begin{quote}
			\begin{description}
			
				\item[(b)] $\st : \scf(\starc) \to \mathc$ is defined to be $\st (x + iy) = \st(x) + i\cdot \st(y)$.
				
			\end{description}
		\end{quote}
		We may also extend to $\scl(\starr)$:
		\begin{quote}
			\begin{description}
				
				\item[(c)] When $x \in \scl(\starr)$ we may extend the $\st$-mapping to $\st : \starr \to \mathr \cup \{\pm\infty\}$ by $\st(x) = \pm \infty$ for $x > 0$ in $\scl(\starr)$ or $x < 0$ in $\scl(\starr)$, respectively.
			
			\end{description}
		\end{quote}

	\end{definition}
	
	We now discuss the asymptotic expansion of the finite numbers.
	
	\begin{theorem}\label{T: Representation}
		Each $z \in \scf(\starc)$ has asymptotic expansion as $z = \st(z) + dz$ where $dz \approx 0$.
	\end{theorem}
	\begin{proof}
		We consider the case where $z \in \scf(\starr)$. Define $dz =: z - \st(z)$ and observe $\st(z)$ exists by virtue of the completeness of $\mathr$.
		
		We now show that $dz \approx 0$. Stated in another way, we show $z \approx \st(z)$.
		
		Consider $S_z = \{ r \in \mathr : r \leq z\}$ and recall that $\st(z) = \sup S_z$. Suppose on the contrary that $\st(z) \not\approx z$, then there ezists $n \in \mathn$ such that $| \st (z) - z | \geq \frac{1}{n}$.
			\begin{description}
			
				\item[Case 1] If $z > \st(z)$, then $|\st(z) - z| = z - \st(z)$ and $z - \st(z) \geq \frac{1}{n}$ which implies $z \geq \st(z) + \frac{1}{n}$. Thus, $\st(z) + \frac{1}{n} \in S_z$ so $\st(z) + \frac{1}{n} \leq \st(z)$ since $\st(z) = \sup S_z$. Therefore, $\frac{1}{n} \leq 0$, a contradiction.
				\item[Case 2] If $z < \st(z)$ then $|\st(z) - z| = \st(z) - z \geq \frac{1}{n}$ which implies $\st(z) - \frac{1}{n} \geq z$. Since $\st(z)$ is the least upper bound for $S_z$ the set $\{ r \in S_z : \st(z) - \frac{1}{n} < r < \st(z) \}$ is non-empty. So there exists an $r_0 \in S_z$ such that $\st(z) - \frac{1}{n} < r_0 < \st(z)$ hence $r_0 > z$, a contradiction since $r_0 \in S_z$.
			
			\end{description}
		Therefore $z \approx \st(z)$ and $dz \approx 0$, so $z = \st(z) + dz$. The case of $z \in \scf(\starc)$ follows from $\st(x + iy) = \st(x) + i\st(y)$.
	\end{proof}

	\begin{corollary}
		Every $z \in \scf(\starc)$ may be \emph{uniquely} represented as $z = c + h$ where $c \in \mathc$ and $h \in \sci(\starc)$.
	\end{corollary}
	\begin{proof}
		The existence of such an expansion follows directly from Theorem \ref{T: Representation} for $c = \st(z)$ and $h = dz$. As to the uniqueness, let $z = c_1 + h_1$ for some $c_1 \in \mathc$ and $h_1 \in \sci(\starc)$. Then $c + h = c_1 + h_1$, implying $c - c_1 = h_1 - h$, hence $c - c_1 \approx 0$. Thus $c - c_1 = 0$ since $0$ is the only infinitesimal in $\mathc$.
	\end{proof}
		
	Perhaps surprisingly, the complex/real numbers may be presented as a factor ring of sets of non-standard complex/real numbers, respectively.
	
	\begin{theorem}
		$\scf(\starc) / \sci(\starc)$ is \emph{field} isomorphic to $\mathc$ under $[z] \to \st(z)$. Similarly, $\scf(\starr) / \sci(\starr)$ is order field isomorphic to $\mathr$ under $[x] \to \st(x)$.
	\end{theorem}
	\begin{proof}
		We shall focus on the real case, leaving the complex case to the reader. Recall that $\scf(\starr)$ is an integral domain and $\sci(\starr)$ is a maximal convex ideal in $\scf(\starr)$ (Lemmas \ref{L: Finite Integral Domain} and \ref{L: Maximal Convex}), thus the factor ring is a field. Recall, $[x] = \{ y \in \scf(\starr) : x - y \in \sci(\starr)\}$.
		
		We need to demonstrate that this function is well-defined, bijective, and a field homomorphism. Indeed, let $[x] = [y]$, then $\st(x) = \st(y)$ since $x - y \approx 0$, thus the function is well-defined. Next, suppose $\st(x) = \st(y)$. Then $x \approx y$, therefore $[x] = [y]$, and the function is injective. Let $r \in \mathr$. Then by viewing $r$ as an embedded element of $\scf(\starr) \subseteq \starr$, we have $\st(r) = r$, hence the function is surjective.
		
		Next we show that the function is a field homomorphism. For addition (with the help of Theorem \ref{T: Representation}) we have $x + y = \st(x) + \st(y) + dx + dy$ and $\st(x + y) = \st(\st(x) + \st(y)) = \st(x) + \st(y)$ since $\st(x), \st(y) \in \mathr$. For multiplication we have $xy = \st(x)\st(y) + \st(x)dy + \st(y)dx + dxdy$ and $\st(xy) = \st(x)\st(y)$. For division, assume $y \not\approx 0$ so that $\frac{x}{y} = z$ exists in $\scf(\starr)$. Then $x = zy$ hence $\st(x) = \st(z)\st(y)$.
		
		As to the preservation of the order relation, for $x_1 < x_2$ we have $\st(x_1) \leq \st(x_2)$ where the equality holds whenever $x_1 \approx x_2$. 
		
		Therefore, the function is a well-defined, bijective, field homomorphism from $\scf(\starr) / \sci(\starr) \to \mathr$, which preserves the order relation.
	\end{proof}
	
	\begin{examples} Having established the $\st$-mapping as a field isomorphism we apply it.
	\begin{quote}
		\begin{description}
		
			\item[(i)] $\st (5 + \rho) = 5$, just as $5 + \rho \approx 5$.
			\item[(ii)] $\st \left (\frac{7 + \rho^5}{8 + \sqrt{\rho}} \right ) = \frac{\st(7 + \rho^5)}{\st(8 + \sqrt{\rho})} = \frac{7}{8}$.
			\item[(iii)] 
				\begin{align}
					\st \left ( \frac{\sqrt{1 + dx} - 1}{dx} \right ) &= \st \left ( \frac{(\sqrt{1 + dx} - 1)(\sqrt{1 + dx} + 1)}{dx (\sqrt{1 + dx} + 1)} \right ) = \st \left (\frac{1 + dx - 1}{dx(\sqrt{1 + dx} + 1)} \right)\notag\\
													&= \st \left (\frac{1}{(\sqrt{1 + dx} + 1)} \right) = \frac{\st(1)}{\st(\sqrt{1 + dx} + 1)}\notag\\
													&= \frac{1}{\st(\sqrt{1 + dx}) + \st(1)} = \frac{1}{1 + 1} = \frac{1}{2}\notag
				\end{align}
			\item[(iv)] Let $dx > 0$. Then, $\st \left ( \frac{1}{dx^2 + dx} \right) = \st \left ( \frac{1}{dx(dx + 1)} \right) = \st \left( \frac{1}{dx} \right) \st \left ( \frac{1}{dx + 1} \right) = \infty \cdot 1 = \infty$.
		
		\end{description}
	\end{quote}
	\end{examples}
	
	For a method of teaching calculus based on the standard part mapping we refer to Keisler \cite{jKeisE} and Todorov \cite{tdTod2000a}.

\section{Non-Standard Extension of a Standard Set}\label{S: NSE of Set}

	We generalize the process of Section \ref{S: A Non-Standard Extension of C} to any set.
	
	\begin{definition}\label{D: Extension of Set}\index{Non-Standard Extension of a Set}
		Let $S \subeq \mathc$ and $(\sep) \in \mathc^{\rplus}$ be a net in $\mathc$. Then the \textbf{non-standard extension} $\stars$ of $S$ is given by $\stars = \{ \bra \sep \ket : \{\varep \in \rplus : \sep \in S\} \in \scu \}$.
	\end{definition}
	
	From the above definition we can see that $\sep \in S$ quite often, but not necessarily for all $\varep \in \rplus$. It just so happens that we can select some representative, say $\bep$, such that $\bep \in S$ for all $\varep \in \rplus$.
	
	\begin{lemma}[Regular Representatives]\label{L: Regular Representatives}
		Let $S \subseteq \mathc$. Then for each $\bra \sep \ket \in \stars$ there exists a net $(\xep) \in \crplus$ such that $\bra \sep \ket = \bra \xep \ket$ and $\xep \in S$ for all $\varep \in \rplus$. We say $(\xep)$ is a \textbf{regular representative} of $\bra \sep \ket$ relative to $S$.
	\end{lemma}
	\begin{proof}
		If $\bra \sep \ket \in \stars$ then $A =: \{ \varep \in \rplus : \sep \in S\} \in \scu$. Let $(\xep) \in \mathr^{\rplus}$ be the net defined by
			\begin{align}
				\xep =	
				\begin{cases}
					\sep &\text{if $\varep \in A$}\\ 
					s \in S \textnormal{ arbitrarily }&\text{if  $\varep \in \mathc \setminus A$}\notag
				\end{cases}
			\end{align}
		Then we have $\bra \sep \ket = \bra \xep \ket$, and $\xep \in S$ for all $\varep \in \rplus$.
	\end{proof}


	\begin{examples}$ $
	\begin{quote}
		\begin{description}
	
			\item[(i)] $\starr$ is the non-standard extension of $\mathr$ with elements $\bra \aep \ket \in \starr$ if $\aep \in \mathr$ a.e.
			\item[(ii)] $\starn$ is the non-standard extension of $\mathn$ with elements $\bra \aep \ket \in \starn$ if $\aep \in \mathn$ a.e.
			\item[(iii)] $\starrplus$ is the non-standard extension of $\mathr_+$ with elements $\bra \aep \ket \in \starrplus$ if $\aep \in \mathr_+$ a.e.
	
		\end{description}
	\end{quote}
	\end{examples}
	
	\begin{theorem}[Boolean Properties]\label{T: Boolean}
		Let $A, B \subseteq \mathc$. Then the following are properties of the non-standard extensions ${^*A}$ and ${^*B}$:
		\begin{quote}
			\begin{description}
			
				\item[(a)] ${^*\varnothing} = \varnothing$
				\item[(b)] ${^*(A \cup B)} = \stara \cup \starb$ 
				\item[(c)] ${^*(A \cap B)} = \stara \cap \starb$
				\item[(d)] ${^*(A \setminus B)} = \stara \setminus \starb$
				\item[(e)] $A \subseteq B$ \iff $\stara \subseteq \starb$
		
			\end{description}
		\end{quote}
	\end{theorem}
	\begin{proof}
		We shall prove \textbf{(e)}, leaving the rest to the reader.
		\begin{description}
		
			\item[($\Rightarrow$)] Let $A \subseteq B$, $A_\varep =: \{ \varep \in \rplus : \zep \in A\}$ and $B_\varep =: \{ \varep \in \rplus : \zep \in B\}$. Then $\bra \zep \ket \in \stara$ \iff $A_\varep \in \scu$. Note that $A_\varep \subseteq \{ \varep \in \rplus : A_\varep \subseteq B_\varep\}$ which implies that $\{ \varep \in \rplus : A_\varep \subseteq B_\varep\} \in \scu$ and therefore $\stara \subseteq \starb$.
			\item[($\Leftarrow$)] Suppose $\stara \subseteq \starb$. Suppose to the contrary that $A \not\subset B$. Then there exists some $a \in A$ such that $a \not\in B$. Taking $\bra \aep \ket = a$ for all $\varep \in \rplus$ we have $\bra \aep \ket \in \starb$, moreover we have $\{ \varep \in \rplus : \aep \in B \} = \rplus$. A contradiction, since we assumed $A \not\subset B$. 
		
		\end{description}
	\end{proof}

\section{Internal Sets}\label{S: Internal Sets}
	
	As it happens, our previous construction of using a net of points in $\mathbb{R}$ to create a single point of $^*\mathbb{R}$ can work for other objects. In this section we discuss this process on a net of subsets of $\mathbb{R}_+$. It should be noted that the non-standard extensions of standard sets from the previous section are a special case of the more general internal sets.

	\begin{definition}\label{D: Internal}\index{Internal Set}
		Let $(\Sep) \in \mathcal{P}(\mathc)^{\rplus}$, and define $\bra \Sep \ket =: \{ \bra \zep \ket \in \starc : \zep \in \Sep$ a.e.$\}$. Sets of the form $\bra \Sep \ket$ are \textbf{internal sets}. The internal set $\bra \Sep \ket$ is \textbf{generated} by the net $(\Sep)$. The set of all internet subsets of $\starc$ is denoted ${\star\mathcal{P}(\mathc)}$. If $S$ is not internal, then it is \textbf{external}. The set of external sets are denoted $\mathcal{P}(\starc) \setminus \star\mathcal{P}(\mathc)$.
	\end{definition}
	
	\begin{example}
		The standard sets $\mathn, \mathn_0, \mathq, \mathr, \rplus, \mathc$ are all external. We wait until the Dedekind Completeness result (Section \ref{S: Completeness}) to prove this.
	\end{example}
	
	\begin{lemma}
		Let $S \subseteq \mathc$. Then $\stars$ is an internal set generated by the constant net $S_\varep = S$ for all $\varep \in \rplus$. That is to say, $\stars = \bra S \ket$.
	\end{lemma}
	\begin{proof}
		$\stars = \{ \bra \zep \ket \in \starc : \zep \in S \textnormal{ a.e.}\}$ by Definition \ref{D: Extension of Set}, $\stars$ is an internal set by Definition \ref{D: Internal}.
	\end{proof}

	So the non-standard extension of a set is a special case of an internal set. We may think of internal sets as a generalization of the process of creating a non-standard extension.

	\begin{example}[Infinitesimal Interval]
		Let $(0, \rho) = \{ x \in \starr : 0 < x < \rho \}$. Take as our net of real subsets $S_\varepsilon = (0, \varepsilon) = \{ x \in \mathbb{R} : 0 < x < \varepsilon\}$, where $\rho = \bra \varep \ket$. Then $\bra S_\varepsilon \ket = (0, \rho)$. Since $\rho$ is an infinitesimal, there is no standard set which could be a subset of $(0, \rho)$, thus it is a proper internal set.
	\end{example}

	\begin{example}[Intervals in $\starr$]
		More generally, let $\alpha, \beta \in \starr$ with $\alpha < \beta$. With the intervals $(\alpha, \beta), [\alpha, \beta], [\alpha, \beta), (\alpha, \beta]$ in $\starr$ defined in the usual sense, each is an internal set.
	\end{example}

	\begin{example}
		We briefly depart from the non-standard real numbers to discuss subsets of the non-standard natural numbers. Recall that $\nu$ is infinitely large from Example \ref{E: Canonical Infinitely Large}. Let $\Omega = \{ 1, 2, 3, ... , \nu\}$, which is to say, $\Omega = \{n \in $$^*\mathbb{N} : n \leq \nu\}$. We claim $\Omega$ is internal and if $\Omega_\varepsilon = \{ n \in \mathbb{N} : n \leq \alpha_\varepsilon$ a.e.$\}$, then $\bra \Omega_\varepsilon \ket = \Omega$.
	
		Indeed, let $\bra a_\varepsilon \ket \in \Omega$. Then $a_\varepsilon \in \mathbb{N}$ a.e. and $a_\varepsilon \leq \alpha_\varepsilon$ a.e. Then $a_\varepsilon \in \Omega_\varepsilon$ a.e., which implies $\bra a_\varepsilon \ket \in \bra \Omega_\varepsilon \ket$.
	
		Conversely, $\bra a_\varepsilon \ket \in \bra \Omega_\varepsilon \ket$ if and only if $a_\varepsilon \in \Omega_\varepsilon$ a.e. which implies $a_\varepsilon \in \mathbb{N}$ a.e. and $a_\varepsilon \leq \alpha_\varepsilon$ a.e. which implies $\bra a_\varepsilon \ket \in $$^*\mathbb{N}$ and $\bra a_\varepsilon \ket \leq \nu$. That is, $\bra a_\varepsilon \ket \in \Omega$. Sets of the form $\Omega$ are called \textbf{hyperfinite}.
	\end{example}
	
	\begin{example}
		Let $r \in \mathr$ and $\nu \in \starn$. Defining ${^*B(r, \frac{1}{\nu}}) =: \{ x \in \starr : r - \frac{1}{\nu} < x < r + \frac{1}{\nu} \}$, we have that ${^*B(r, \frac{1}{\nu}})$ is an internal set. Indeed, $\nu = \bra \nu_\varep \ket$ and ${^*B(r, \frac{1}{\nu}}) = \bra B(r, \frac{1}{\nu_\varep}) \ket$.
	\end{example}
	
\section{Completeness of $\starr$}\label{S: Completeness}
	Our focus turns to the completeness of $\starr$. First, we discuss \textbf{Dedekind Completeness} which is the non-standard analogue of the Supremum Principle in $\mathr$. We apply Dedekind Completeness to discuss the \textbf{Spilling Principles}. Next we present two versions of the \textbf{Saturation Principle}; one for open intervals in $\starr$ and one for internal sets in $\starr$, and as a direct corollary to the open interval version of the Saturation Principle we obtain the \textbf{Cantor Completeness}. 
	
	\begin{lemma}\label{L: Order Completeness Lemma}
		Let $\bra \Aep \ket \subseteq \starr$ be an internal set that is bounded from above by $\bra \bep \ket \in \starr$. That is, for all $\alpha \in \bra \Aep \ket$, $\alpha \leq \bra \bep \ket$. Then, $J =: \{ \varep \in \rplus :$ For all $x \in \Aep$, $x \leq \bep\} \in \scu$.
	\end{lemma}
	\begin{proof}
		Suppose to the contrary that $J = \{ \varep \in \rplus :$ for all $x \in \Aep$, $x \leq \bep\} \notin \scu$. Since $\scu$ is an ultrafilter, $\rplus \setminus J =: \{ \varep \in \rplus :$ there exists $x \in \Aep$, such that $x > \bep\} \in \scu$. Let 
			\begin{equation}\label{E: X Epsilon}
				\Xep = \{ x \in \Aep : x > \bep\}
			\end{equation}
		and observe that $K =: \{ \varep \in \rplus : \Xep \not= \varnothing\} \in \scu$. We define $(\aep) \in \mathr^{\rplus}$ by $\aep \in \Xep$ if $\varep \in K$ and by $\aep = 1$ if $\varep \in \rplus \setminus K$ (Such a definition requires the Axiom of Choice). Let $\alpha = \bra \aep \ket$, then $\aep \in \Aep$ a.e. and $\aep > \bep$ a.e. and so by the definition of $\Xep$ in (\ref{E: X Epsilon}) we have $\alpha \in \bra \Aep \ket$ and $\alpha > \bra \bep \ket$, a contradiction that $\bra \Aep \ket$ is bounded from above by $\bra \bep \ket$.
	\end{proof}

	\begin{theorem}[Dedekind (Order) Completeness in $\starr$]\label{T: Order Completeness}\index{Supremum Completeness}
		If an internal subset $\sca$ of $\starr$ is bounded from above, then $\sup \sca$ in $\starr$ exists.
	\end{theorem}
	\begin{proof}
		By assumption there exists $\beta \in \starr$ such that for all $\alpha \in \sca$, $\alpha \leq \beta$. We have $\sca = \bra \Aep \ket$ and $\beta = \bra \bep \ket$ for some $(\Aep) \in \mathcal{P}(\mathr)^{\rplus}$ and $(\bep) \in \mathr^{\rplus}$, respectively. From Lemma \ref{L: Order Completeness Lemma} we have $J =: \{ \varep \in \rplus :$ for all $x \in \Aep$, $x \leq \bep\} \in \scu$. Define $(\sep) \in \mathr^{\rplus}$ by 
			\begin{align}
				\sep =	
					\begin{cases}
						\sup (\Aep) &\text{if $\varep \in J$}\\ 
						1 &\text{if  $\varep \in \rplus \setminus J$}\notag
					\end{cases}
			\end{align}
		We now show that $\bra \sep \ket$ is an upper bound of $\sca$. Let $\bra \aep \ket \in \sca$ and denote $L =: \{ \varep \in \rplus : \aep \in \Aep \}$. We have $L \in \scu$ hence $J \cap L \in \scu$ (\textbf{(b)} from Definition \ref{D: Filter and Free Filter}). Thus $\{ \varep \in \rplus : \aep \leq \sep \} \in \scu$ since $J \cap L \subseteq \{ \varep \in \rplus : \aep \leq \sep \}$ by the definition of $(\sep)$.
		
		Lastly, to show that $\bra \sep \ket$ is the \emph{least} upper bound, we show that for every $\bra \aep \ket \in \sca$, the set $\{ x \in \sca : \bra \aep \ket < x < \bra \sep \ket \}$ is non-empty. Define $\Sep = \{ x \in \rplus : \aep < x < \sep \}$ if $\varep \in J \cap L$ and arbitrarily (say, $\Sep = \rplus$) if $\varep \in \rplus \setminus (J \cap L)$. Then by the definition of the supremum in $\mathr$, $\Sep \not= \varnothing$ for all $\varep \in \rplus$ since $\sep = \sup \Aep$ for $\varep \in J$. Define $(\xep) \in \mathr^{\rplus}$ (using the Axiom of Choice) by $\xep \in \Sep$ if $\varep \in J \cap L$ and arbitrarily (say $\xep = 1$) for $\varep \in \rplus \setminus (J \cap L)$. We conclude that $\bra \aep \ket < \bra \xep \ket < \bra \sep \ket$ as desired.
	\end{proof}
		
	\begin{corollary}
		The standard sets $\mathn, \mathn_0, \mathq, \mathr, \rplus, \mathc$ are all external sets.
	\end{corollary}
	\begin{proof}
		We shall only prove that $\mathr$ is external. Suppose on the contrary that it is internal. Since $\mathr \subset \starr$ we have that $\mathr$ is bounded from above by any infinitely large positive element of $\starr$. By Order Completeness (Theorem \ref{T: Order Completeness}), $\sup \mathr$ exists in $\starr$, and is infinitely large as an upper bound for $\mathr$. As $\sup \mathr - 1$ is also infinitely large, it is also an upper bound for $\mathr$. This is a contradiction since $\sup \mathr$ is the \emph{least} upper bound. Therefore, $\mathr$ is external.
	\end{proof}

	\begin{corollary}[Spilling Principles]\label{C: Spilling Principles}\index{Spilling Principles}
		If $S \subseteq \starr$ is an internal set, then:
	\begin{quote}
		\begin{description}
		
			\item[(a)] \textbf{Overflow of $\scf(\starr)$:} If $S$ contains arbitrarily large finite positive numbers, then $S$ contains arbitrarily small infinitely large positive numbers.
			\item[(b)] \textbf{Underflow of $\scf(\starr)$:} If $S$ contains arbitrarily small finite non-infinitesimal numbers, then $S$ contains arbitrarily large positive infinitesimals.
			\item[(c)] \textbf{Underflow of $\scl(\starr)$:} If $S$ contains arbitrarily small infinitely large positive numbers, then $S$ contains arbitrarily large finite positive numbers.
			\item[(d)] \textbf{Overflow of $\sci(\starr)$:} If $S$ contains arbitrarily large positive infinitesimals, then $S$ contains arbitrarily small finite non-infinitesimal positive numbers.
		
		\end{description}
	\end{quote}
	\end{corollary}
	\begin{proof} 
		We begin with \textbf{(a)}. Suppose there exist arbitrarily large finite positive numbers in $S$. In the case that $S$ is unbounded from above we may conclude that $S$ contains an infinitely large number.
		
		In the case where $S$ is bounded from above, we know from Theorem \ref{T: Order Completeness} (Dedekind Completeness) that $\alpha = \sup (A)$ exists in $\starr$. Furthermore, by the assumptions made on $S$, $\alpha$ is an infinitely large positive number. Since $\alpha = \sup (A)$, there exists $x \in S$ such that $\frac{\alpha}{2} < x < \alpha$ and since $\frac{\alpha}{2}$ is infinitely large, $x$ must be as well.
		
		Now that we have shown that the set of positive infinitely large numbers in $S$ is non-empty ($\scl(S_+) \not= \varnothing$), we show that it does not have a lower bound. Suppose to the contrary that $l \leq x$ for some $l \in \scl(\starr_+)$ and all $x \in \scl(S_+)$. Notice the set $S_l = \{ x \in S : 0 \leq x < l \}$ is internal since $S_l = S \cap (0 , l)$. Since $S \cap \scf(\starr_+) \subseteq S_l$ we have $\scl(S_l) \not=\varnothing$ by the same argument from before. So there exists an inifinitely large $x \in S$ with $0 < x < l$, a contradiction of the choice of $l$.
		
		Therefore there exist arbitrarily small infinitely large positive numbers in $S$. We leave \textbf{(b)} - \textbf{(d)} to the reader.
	\end{proof}

	\begin{theorem}[Saturation Principle for nested open intervals]\label{T: Interval Saturation}\index{Saturation Principle!Open Intervals}
		Every nested sequence of open intervals in ${^*}\mathbb{R}$ has a non-empty intersection.
	\end{theorem}
	\begin{proof} 
		Given $\alpha_n, \beta_n \in {^*}\mathbb{R}$ such that $\alpha_n < \beta_n$ for all $n \in \mathn$, and $\alpha_n \leq \alpha_{n+1} < \beta_{n+1} \leq \beta_n$, we have to show that there exists $\gamma \in {^*}\mathbb{R}$ such that $\alpha_m < \gamma < \beta_m$ for all $m \in \mathbb{N}$. Let $\alpha_n = \bra a_{n\varepsilon} \ket$ and $\beta_n = \bra b_{n\varepsilon} \ket$. From $\alpha_n < \beta_n$ for all $n \in \mathbb{N}$, if we define $A_n = \{ \varepsilon \in \mathbb{R}_+ : a_{n\varepsilon} < b_{n\varepsilon} \}$, then $A_n \in \mathcal{U}$.
	
		For $n=1$, without loss of generality we can assume that $A_1 = \mathbb{R}_+$. Indeed, suppose not, then $\mathbb{R}_+ \setminus A_1 \not= \varnothing$. Redefine:
		\begin{align}
			a'_{1\varepsilon} &= a_{1\varepsilon}\notag\\
			b'_{1\varepsilon} &= 	\begin{cases} 	b_{1\varepsilon} & \text{ if $\varepsilon \in A_1$,}
											\\
										a_{1\varepsilon} + 1 & \text{ if $\varepsilon \in \mathbb{R}_+ \setminus A_1$}
							\end{cases}\notag
		\end{align}
		Then we have $a'_{1\varepsilon} < b'_{1\varepsilon}$ for all $\varepsilon \in \mathbb{R}_+$. Also, $\alpha_1 = \bra a'_{1\varepsilon} \ket$ and $\beta_1 = \bra b'_{1\varepsilon} \ket$. So $A'_1 = \{ \varepsilon \in \mathbb{R}_+ : a'_{1\varepsilon} < b'_{1\varepsilon} \} = \mathbb{R}_+$. Thus justifying our assumption that $A_1 = \rplus$.

		Henceforth, we will keep the original notation, using $a_{1\varepsilon}$ and $b_{1\varepsilon}$ without the primes. So we assume that $a_{1\varepsilon} < b_{1\varepsilon}$ for all $\varepsilon \in \mathbb{R}_+$.

		Next, we define $\mu : \mathbb{R}_+ \to \mathbb{N} \cup \{ \infty \}$ by $\mu(\varepsilon) := \max\{m \in \mathbb{N} : \bigcap_{n=1}^{m}(a_{n\varepsilon}, b_{n\varepsilon}) \not= \varnothing\}$. Notice that $\mu(\varepsilon)$ is well defined because of our assumption that $A_1 = \rplus$.
	
		Next, choose $(c_\varepsilon) \in \mathbb{R}^{\mathbb{R}_+}$ such that $c_\varepsilon \in \bigcap_{n=1}^{\mu(\varepsilon)} (a_{n\varepsilon}, b_{n\varepsilon})$. Notice that the intersection is non-empty for all $\varepsilon \in \mathbb{R}_+$ by our choice of $\mu(\varepsilon)$. Define $\gamma \in {^*}\mathbb{R}$ by $\gamma = \bra c_\varepsilon \ket$.
	
		And now, for the heart of the proof. Let $m \in \mathn$. We must show that $\alpha_m < \gamma < \beta_m$, that is, that $\amep < \cep < \bmep$ a.e.
		
		Notice that $\{ \varep \in \rplus : \bigcap_{n =1}^{\mu(\varep)} (\anep, \bnep) \not= \varnothing \} = \rplus$ by definition of $\mu(\varep)$ and so it is in $\scu$. Consider $S_m = \{ \varep \in \rplus : \bigcap_{n=1}^{m} (\anep, \bnep) \not= \varnothing\}$. Since $\bigcap_{n=1}^{m} (\alpha_n, \beta_n) \not= \varnothing$ (in $\starr$), we know that $S_m \in \scu$.
		
		We now claim $S_m \subseteq \{\varep \in \rplus : \amep < \cep < \bmep\}$. Indeed, $\varep \in S_m \Leftrightarrow \bigcap_{n=1}^{m} (\anep, \bnep) \not= \varnothing$. This implies that $m \leq \mu(\varep)$, by the definition of $\mu(\varep)$, which implies $\cep \in (\amep, \bmep)$ (because $\cep \in \bigcap_{n=1}^{\mu(\varep)} (\anep, \bnep) \subseteq (\amep, \bmep)$) $\Leftrightarrow \varep \in \{\varep \in \rplus : \amep < \cep < \bmep\}$. By \textbf{(c)} of Definition \ref{D: Filter and Free Filter}, $\amep < \cep < \bmep$ a.e. for all $n \in \mathn$ since $m$ was arbitrary.
		
		Therefore, $\gamma = \bra \cep \ket \in \starr$ is an element of $\bigcap_{n=1}^{\infty}(\alpha_n, \beta_n)$, making the intersection non-empty.
	\end{proof}
		
	\begin{corollary}[Cantor Principle in $\starr$]\index{Cantor Principle}
		Every nested sequence of closed intervals in $\starr$ has a non-empty intersection.
	\end{corollary}
	\begin{proof}
		As every closed interval in $\starr$ contains an open interval in $\starr$, by Theorem \ref{T: Interval Saturation}, every nested sequence of closed intervals in $\starr$ has a non-empty intersection.
	\end{proof}
		
	We now present the general version of the Saturation Principle. In the following $c^+$ denotes the cardinal number that is the successor of $c = \card(\mathr)$; notice that $\card(\rplus) = c$, as well. Recall that a collection of sets $\{S_\gamma\}_{\gamma \in \Gamma}$ has the \textbf{finite intersection property} if $\cap_{\gamma \in F} S_\gamma \not= \varnothing$ for each finite subset $F \subseteq \Gamma$. 
		
	\begin{theorem}[General Saturation Principle]\label{T: General Saturation}\index{Saturation Principle!General}
		The non-standard complex numbers $\starc$ are $c^{+}$-saturated in the sense that every family $\{S_\gamma\}_{\gamma \in \Gamma}$ of internal sets of $\starc$ with the finite intersection property and with $\card(\Gamma) \leq c$ has a non-empty intersection.
	\end{theorem}
	
	We omit the proof, but note that it is similar to the proof of Theorem \ref{T: Interval Saturation} but utilizes more complex combinatorial arguments. 
	
	\begin{remark}
		Observe that $\mathn, \mathq, \mathr, \rplus, \mathr^{n}, \mathc, \mathc^{n}, \mathcal{D}(\mathr^{d})$ all have cardinality at most $c$ and so may be used as the index set in Theorem \ref{T: General Saturation}. Here $\mathcal{D}(\mathr^{d})$ denotes the class of $C^\infty$-functions from $\mathr^d$ to $\mathc$ with compact support.
	\end{remark}
	
\section{Non-Standard Extension of a Function}\label{S: Functions}
	Just as a standard set may be extended to a non-standard set, we may extend an ordinary function to a non-standard function.
	
	\begin{definition}\label{D: NSE of Function}\index{Non-Standard Extension of a Function}
		Let $f: X \to \mathc$ where $X \subseteq \mathc$, be a (standard) function, and let $\bra \xep \ket \in \starx$. Without loss of generality we can assume that $\xep \in X$ for all $\varep \in \rplus$ by Lemma \ref{L: Regular Representatives}. We define the non-standard extension $\starf: \starx \to \starc$ by the formula $\starf(\bra \xep \ket) = \bra f(\xep) \ket$.
	\end{definition}
	
	\begin{remark}
		We drop the asterisks in front of $f$ for well known functions such as $e^x$, $\ln x$, etc.
	\end{remark}
	
	\begin{theorem}[Properties of Non-Standard Extension of $f$]\label{T: Properties of f-star} Consider $f$ and $\starf$ as in the above definition. Then,
	\begin{quote}
		\begin{description}
		
			\item[(a)] $\starf$ is an extension of $f$ in that $\starf \mid_{X} = f$ (Notice $X \subseteq \starx$ by Definition \ref{D: Extension of Set})
			\item[(b)] $\dom(\starf) = $ $^*(\dom(f))$, and
			\item[(c)] $\ran(\starf) = $ $^*(\ran(f))$
		
		\end{description}
	\end{quote}
	\end{theorem}
	\begin{proof} $ $
	\begin{quote}
		\begin{description}
		
			\item[(a)] $\bra \xep \ket \in X$ \iff there exists $x \in X$ such that $\{ \varep \in \rplus : \xep = x\} \in \scu$ \iff $\bra \xep \ket = \bra x \ket$. So $\starf(\bra \xep \ket) = \starf(\bra x \ket) = \starf(x) = f(x)$.
			\item[(b)] $\dom(\starf) = \starx = $ $^*(\dom(f))$.
			\item[(c)] $\bra \yep \ket \in \ran(\starf)$ \iff there exists $\bra \xep \ket \in \dom(\starf)$ such that $\starf(\bra \xep \ket) = \bra \yep \ket$ \iff $\{ \varep \in \rplus : f(\xep) = \yep\} \in \scu$. We have $\{ \varep \in \rplus : f(\xep) = \yep\} \subseteq \{\varep \in \rplus : \yep \in \ran(f)\}$, so $\{\varep \in \rplus : \yep \in \ran(f)\} \in \scu$. The last statement is equivalent to $\bra \yep \ket \in $ $^*(\ran(f))$. 
		
		\end{description}
	\end{quote}
	\end{proof}
	
	We now present some basic examples to fortify Definition \ref{D: NSE of Function}.
	
	\begin{examples}\label{E: NSF}$ $
	\begin{quote}
		\begin{description}
		
			\item[(i)] Consider $\ln x$ with $\dom(\ln x) = \rplus$ and $\ran(\ln x) = \mathr$. For the non-standard extension $\star\ln x$ we have $\dom(\star\ln x) = \starr_+$ and $\ran(\star\ln x) = \starr$ by Theorem \ref{T: Properties of f-star}. 
			
			Let $\bra \xep \ket \in \starr_+$. Without loss of generality (Lemma \ref{L: Regular Representatives}) we may assume that $\xep \in \rplus$ for all $\varep \in \rplus$. Therefore, $\star\ln \bra \xep \ket = \bra \ln \xep \ket$.
			\item[(ii)] Next, consider $e^x$ with $\dom(e^x) = \mathr$ and $\ran(e^x) = \rplus$. Once again, from Theorem \ref{T: Properties of f-star}, the non-standard extension $\star e^x$ has $\dom(\star e^x) = \starr$ and $\ran(\star e^x) = \starr_+$. The values may be computed via Definition \ref{D: NSE of Function} so that $\star e^{\bra \xep \ket} = \bra e^{\xep} \ket$. Notice we need not invoke Lemma \ref{L: Regular Representatives} since $\dom(e^x)$ has no peculiarities.
			\item[(iii)] The extensions of all trigonometric functions ($\sin$, $\cos$, etc.) are similar. In particular, consider $\sin x$ with $\dom(\sin x) = \mathr$ and $\ran(\sin x) = [-1, 1]$. We have $\dom(\star\sin x) = \starr$ and $\ran(\star\sin x) = \star[-1,1] \subset \starr$ with values $\star\sin \bra \xep \ket = \bra \sin \xep \ket$.
		
		\end{description}
	\end{quote}
	\end{examples}

\section{Transfer Principle by Example}\label{S: Transfer by Example}

Remember that the real numbers may be arrived at constructively with Dedekind Cuts whereby the Completeness of the reals is a theorem, yet there is the alternate approach of axiomatically defining the real numbers whereby the Completeness is an axiom. The same holds true for the non-standard reals. Up until now we constructed the non-standard reals using ultrafilters and an equivalence relation, but there is the alternative of an axiomatic approach. The Transfer Principle is one of the key axioms which allows us to connect the standard analysis to the non-standard analysis.

We proceed by recalling theorems proven in this paper using ultrafilters and the equivalence relation on $\mathc^{\rplus}$, and noting that formalizing the equivalent standard result, and putting asterisks in the proper places leads us to the non-standard result.

	\begin{example}
	In Theorem \ref{T: Field} we proved that the non-standard complex numbers $\starc$ is a field. We proved our assertion using the framework of ultrafilters, but this needn't be the case. Recall that $\starc = \mathc^{\rplus} / \sim$ and so $\starc$ is already a ring. So the assertion that $\starc$ is a field is reduced to saying that each element has a multiplicative inverse. This is formalized as: 
	\begin{equation}\label{E: Formal Field}
		(\forall z \in \starc)(\exists w \in \starc)[zw = 1].
	\end{equation}
	We now notice that (\ref{E: Formal Field}) may be obtained by replacing $\mathc$ with $\starc$ in the standard statement: $$(\forall z \in \mathc)(\exists w \in \mathc)[zw = 1],$$
	which is true since $\mathc$ is a field.
	\end{example}

	\begin{example}
	In Theorem \ref{T: Algebraically Closed} we proved that $\starc$ is an algebraically closed field. In the proof we relied heavily upon the equivalence classes of $\starc$ and the fixed ultrafilter. The formalization of Theorem \ref{T: Algebraically Closed} is:
	\begin{equation}\label{E: Formal Closed}
		(\forall P \in \starc[x])(\exists z \in \starc)[P(z) = 0].
	\end{equation}
	Notice (\ref{E: Formal Closed}) is obtained by replacing $\mathc[x]$ with $\starc[x]$ and $\mathc$ with $\starc$ in the standard statement: $$(\forall P \in \mathc[x])(\exists z \in \mathc)[P(z) = 0],$$
	which is true since $\mathc$ is an algebraically closed field.
	\end{example}

	\begin{example}
	In Theorem \ref{T: Totally Ordered Field} we proved, using the properties of ultrafilters, that $\starr$ was a totally ordered field. The formalization of the statement for \ref{T: Totally Ordered Field} is:
	\begin{equation}\label{E: Formal Total}
		(\forall x, y \in \starr)[(x > y) \vee (x > y) \vee (x = y)].
	\end{equation}
	Notice (\ref{E: Formal Total}) may be obtained by replacing $\mathr$ with $\starr$ in the standard statement: $$(\forall x, y \in \mathr)[(x > y) \vee (x > y) \vee (x = y)],$$
	which is true since $\mathr$ is a totally ordered field.
	\end{example}

	\begin{example}
	In Theorem \ref{T: Properties of f-star} we will show $\dom(\starf) = {\star(\dom(f))}$. Recall the definition of the domain $\dom(f)$ may be formalized as:
	\begin{equation}\label{E: Formal dom}
		(\forall x \in \mathr)[ x \in \dom(f) \Leftrightarrow (\exists y \in \mathc)[f(x) = y]].
	\end{equation}
	Note that replacing $\mathr$ with $\starr$, $\mathc$ with $\starc$, $\dom(f)$ with ${\star(\dom(f))}$ and $f$ with $\starf$ in (\ref{E: Formal dom}) we obtain: $$(\forall x \in \starr)[x \in {\star(\dom(f))} \Leftrightarrow (\exists y \in \starc)[\starf(x) = y]].$$ We now note that the right side of the if and only if statement is what we mean by $x \in \dom(\starf)$, and therefore ${\star(\dom(f))} = \dom(\starf)$ as was already proven.
	\end{example}

	\begin{example}
	In Theorem \ref{T: Properties of f-star} we will show $\ran(\starf) = {\star(\ran(f))}$. Recall the definition of the range $\ran(f)$ may be formalized as:
	\begin{equation}\label{E: Formal ran}
		(\forall y \in \mathc)[ y \in \ran(f) \Leftrightarrow (\exists x \in \dom(f))[f(x) = y]].
	\end{equation}
	Note that by replacing $\mathc$ with $\starc$, $\ran(f)$ with ${\star(\ran(f))}$, $\dom(f)$ with ${\star(\dom(f))}$, and $f$ with $\starf$ in (\ref{E: Formal ran}) we obtain: $$(\forall y \in \starc)[ y \in {\star(\ran(f))} \Leftrightarrow (\exists x \in {\star(\dom(f))}[\starf(x) = y]].$$ As ${\star(\dom(f))} = \dom(\starf)$, the right side of the if and only if statement is what we mean by $x \in \ran(\starf)$, and therefore ${\star(\ran(f))} = \ran(\starf)$ as was already proven.
	\end{example}

From just these five examples we see that by placing asterisks in the right location we arrive at a non-standard result which we have already proven "the hard way" using ultrafilters. This is not simply a coincidence; indeed, we are hinting at a powerful general theorem which asserts the validity of a standard result if and only if the non-standard result is likewise valid; this is the Transfer Principle. Before we describe the Transfer Principle in any detail, we present a cautionary example.

	\begin{warning}
	The Completeness of the real numbers may be stated as follows: $$\textnormal{Any bounded subset of $\mathr$ has a supremum.}$$ Casually putting asterisks as we have done before yields the statement: $$\textnormal{Any bounded subset of $\starr$ has a supremum.}$$ Whereas the standard result is certainly true, the "transferred" statement is \emph{not} true. As a counterexample consider the set $\scf(\starr) \subset \starr$. It is bounded by any infinitely large element of $\starr$, yet $\scf(\starr)$ does not have a supremum (as demonstrated earlier).
	\end{warning}
	
	In order to rigorously state the Transfer Principle, we must clearly define the mathematical language with which we shall speak.
	
\section{Logical Basis for Transfer}\label{S: Logic for Transfer}

The Transfer Principle has quite an extensive logical foundation, and for the purposes of this text, we shall only take what is absolutely necessary for the \emph{statement} of the Transfer Principle. We shall not concern ourselves with the details of the proof (save but to cite it). We begin with a discussion of our framework.

	\begin{definition}[Superstructure]\label{D: Superstructure}\index{Superstructure}
	Let $S$ be an infinite set. The \textbf{superstructure}, denoted $V(S)$, on $S$ is the union $$V(S) = \bigcup_{n = 0}^{\infty} V_n(S),$$ where the $V_n(S)$ are inductively defined by $V_0(S) = S$, $V_1(S) = S \cup \mathcal{P}(S)$, ..., $V_{n+1}(S) = V_n(S) \cup \mathcal{P}(V_n(S))$.
	
	Elements of $V(S)$ are \textbf{mathematical objects}, elements of $V(S) \setminus S$ are \textbf{sets}, and elements of $S$ are \textbf{individuals}. We do not think of individuals as sets. To understand this, imagine we are in the superstructure $V(\mathr)$; we have $5 \in \mathr$ and $\pi \in \mathr$, but we would never say $5 \in \pi$.
	\end{definition}

	\begin{theorem}[Properties of $V_n(S)$ and $V(S)$]
	$ $
		\begin{quote}
		\begin{description}
		
		\item[(a)] We have $S = V_0(S) \subset V_1(S) \subset ... \subset V(S)$, which implies $\card(S) < \card(V_1(S)) < ... < \card(V(S))$. Similarly, $S = V_0(S) \in V_1(S) \in ... \in V(S)$.
		\item[(b)] Each $V_n(S)$ is \textbf{transitive} in the superstructure in the sense that $A \in V_n(S)$ implies either $A \in S$ ($A$ is an individual) or $A \subset V_n(S)$ ($A$ is a set). Written another way, $V_n(S) \subseteq S \cup \mathcal{P}(V_n(S))$, for all $n \in \mathn$.
		\item[(c)]We also have that the superstructure itself is transitive, which is to say that if $A \in V(S)$, then either $A \in S$ or $A \subset V(S)$, that is, $V(S) \subseteq S \cup \mathcal{P}(V(S))$.
		
		\end{description}
		\end{quote}
	\end{theorem}
		
	The transitivity of the superstructure is the \emph{key} property of the superstructure. It tells us that everything in the superstructure is either an individual or a set. This allows us to apply all of our set theoretic operations ($\cup$, $\cap$, $\setminus$, etc) on the non-individual objects of the superstructure.
	
	We claim that the superstructure $V(\mathr)$ contains all of the objects from analysis that we need. To give an idea of the breadth of the superstructure we demonstrate that the Cartesian product $\mathr \times \mathr$ is in the superstructure, and that by induction $\mathr^n$ is in the superstructure. From there we are able to show that functions, and the algebraic operations are members of the superstructure.

	In order to demonstrate that the Cartesian product $\mathr \times \mathr$ is in the superstructure we define the ordered pair to be $\bra a, b \ket =: \{ \{a\}, \{a, b\}\}$. If we can use this definition to prove the following lemma, then we will have justified our definition.
	
	\begin{lemma}
		Let $a, b \in \mathr$. $\bra a, b \ket = \bra a', b' \ket \Leftrightarrow a = a'$ and $b = b'$.
	\end{lemma}
	\begin{proof} $\bra a, b \ket = \bra a', b' \ket \Leftrightarrow \{\{a\}, \{a, b\}\} = \{\{a'\}, \{a', b'\}\}$.
		\begin{description}
		
		\item[($\Rightarrow$)]
		
			\begin{description}
			
			\item[Case 1:] $\{a\} = \{a'\}$ and $\{a, b\} = \{a', b'\}$ implies $a = a'$ implies $\{a, b\} = \{a, b'\}$ implies $b = b'$.
			\item[Case 2:] $\{a\} = \{a', b'\}$ and $\{a, b\} = \{a'\}$. The former implies $a = a' = b'$ and the latter implies $a = b = a'$ as desired.
			
			\end{description}
			
		\item[($\Leftarrow$)] Trivial. 
		
		\end{description}
	\end{proof}
		
From the above lemma we have established that ordered pairs may be represented as sets. With respect to the superstructure we have $\{a\} \in \mathcal{P}(\mathr) \subseteq V_1(\mathr)$ and $\{a, b\} \in \mathcal{P}(\mathr) \subseteq V_1(\mathr)$. So $\{\{a\}, \{a, b\}\} \in \mathcal{P}(V_1(\mathr)) \subset V_2(\mathr)$. This puts the ordered pair, $\bra a, b \ket \in V_2(\mathr) \subset V(\mathr)$ and therefore $\mathr^2 \subseteq V_2(\mathr)$, hence $\mathr^2 \in V_3(\mathr) \subset V(\mathr)$.

We may now define ordered triples as $\bra a, b, c \ket =: \bra \bra a, b \ket, c \ket$. We may inductively define $n$-tuples in this way as well. This allows us to say that $\mathr^n$ is in the superstructure for any $n \in \mathn$. As a consequence, $\mathc$ is in the superstructure since $\mathc = \mathr^2$.

As mentioned previously, functions are in the superstructure. Indeed, we may think of the function $f$ \emph{as the graph} which is a subset of $\mathr \times \mathr$ which is certainly in the superstructure. In other words, $f \subseteq \mathr^2 \subseteq V_2(\mathr)$, hence $f \in V_3(\mathr)$.

We may also think of addition, subtraction, multiplication, and division as being in the superstructure since we may define each as sets of ordered triples. The breadth of the superstructure is apparent.

We now discuss the language of the superstructure, denoted $\mathcal{L}(S)$. Do not forget that the goal of our language is to be able to formulate statements, and know exactly where to put the asterisks when we wish to transfer them to their non-standard form.

	\begin{definition}[Alphabet]\label{D: Alphabet}\index{Language!Alphabet}
		Let $S$ be an infinite set and $V(S)$ be the superstructure of $S$. The \textbf{alphabet} $\mathcal{A} \cup \mathcal{B} \cup \mathcal{C}$ of the language $\mathcal{L}(S)$ consists of the sets:
		
		\begin{quote}
		\begin{description}
		
		\item[(a)] The set $\mathcal{A}$ is composed of the \textbf{symbols}: $$= \textnormal{ } \in \textnormal{ } \neg \textnormal{ } \wedge  \textnormal{ }( \textnormal{ } ) \textnormal{ }[ \textnormal{ } ] \textnormal{ } \{ \textnormal{ } \} \textnormal{ } \bra \textnormal{ } \ket \textnormal{ } \exists \textnormal{ } \upharpoonright,$$
		\item[(b)] The set $\mathcal{B}$ is composed of countably many \textbf{variables}: $$x, y, z, x_1, x_2, x_3, ..., X, Y, Z, X_1, X_2, X_3, ...$$
		\item[(c)] The set $\mathcal{C} = V(S)$ is the superstructure discussed previously. The members of $\mathcal{C}$ are called \textbf{constants} of the language $\mathcal{L}(S)$.
		
		\end{description}
		\end{quote}
		
		Members of the alphabet $\mathcal{A} \cup \mathcal{B} \cup \mathcal{C}$ are \textbf{letters}.
	\end{definition}
	
	\begin{definition}[Vocabulary]\label{D: Vocabulary}\index{Language!Vocabulary}
		The vocabulary of $\mathcal{L}(S)$ consists of \textbf{words, terms, predicates} and \textbf{propositions} defined recursively as follows:
		
		\begin{quote}
		\begin{description}
		
		\item[(a)] A \textbf{word} is any finite sequence of letters.\index{Language!Word}
		\item[(b)] A word $\tau$ is a \textbf{term} if there exists a finite sequence of words $\tau_1, \tau_2, ..., \tau_m$ such that $\tau = \tau_m$ and for each $n \in \{1, 2, ..., m\}$ we have:\index{Language!Term}
		\begin{quote}
		\begin{description}
		
		\item[(1)] $\tau_n$ is either a variable \textbf{or}
		\item[(2)] $\tau_n$ is a constant \textbf{or}
		\item[(3)] $\tau_n = \bra \tau_i, \tau_j \ket$ for some $i, j < n$ (i.e. an ordered pair of variables and/or constants) \textbf{or}
		\item[(4)] $\tau_n = (\tau_i \upharpoonright \tau_j)$ for some $i, j < n$.
		
		\end{description}
		\end{quote}
		
		A term containing no variables is a \textbf{closed term}. The set of closed terms in $\mathcal{L}(S)$ is denoted by $\mathcal{T}(S)$.\index{Language!Closed Term}
		
		\end{description}
		
		We should take the time to define the symbol $\upharpoonright$. It is essentially function evaluation. For $A \in V(S) \setminus S$ and $a \in V(S)$ we define
		\begin{align}
			A \upharpoonright a =	
				\begin{cases}
					b &\text{if ($\exists ! b \in V(S)$)[$\bra a, b \ket \in A$]}\\ 
					\varnothing &\text{otherwise.}\notag
				\end{cases}
		\end{align}
		
	\begin{remark}\label{R: Evaluation}
		Let $f \in V(S) \setminus S$ be a function. Let $a \in V(S)$ and $x \in \mathcal{B}$. We shall write $f(a)$ and $f(x)$ instead of $f \upharpoonright a$ and $f \upharpoonright x$, respectively.
	\end{remark}
		
		\begin{description}
		
		\item[(c)] A word $P$ is a \textbf{predicate} if there exists a finite sequence of words $P_1, P_2, ..., P_m$ such that $P = P_m$ and for each $n \in \{1, 2, ..., m\}$ we have:\index{Language!Predicate}
		\begin{quote}
		\begin{description}
		
		\item[(1)] $P_n$ is either $( \sigma = \tau )$ for some terms $\sigma$ and $\tau$, \textbf{or}
		\item[(2)] $P_n$ is $(\sigma \in \tau)$ for some terms $\sigma$ and $\tau$, \textbf{or}
		\item[(3)] $P_n = \neg P_i$ for some $i < n$, \textbf{or}
		\item[(4)] $P_n = (P_i \wedge P_j)$ for some $i, j < n$, \textbf{or}
		\item[(5)] $P_n = (\exists x \in A)P_i$ for some $i < n$ and for some term $A$ where the variable $x$ does not occur.
		
		\end{description}
		\end{quote}
				
		\item[(d)] A variable $x$ in the predicate $P$ is \textbf{bounded} if we have the form $$(\exists x \in A)R(x) \textnormal{  \textbf{or}  } (\forall x \in A)R(x),$$ for some predicate $R(x)$ and for some term $A$ in which $x$ does not occur. If $x_1, x_2, ..., x_n$ are all free variables in $P$, then we write $P(x_1, x_2, ..., x_n)$ instead of $P$.\index{Language!Bounded Variable}
		\item[(f)] A \textbf{proposition} is a predicate in which no variable is free. The set of all propositions in $\mathcal{L}(S)$ is denoted by $\Pi(S)$.\index{Language!Proposition}
		
		\end{description}
		\end{quote}
		
	\end{definition}

	\begin{remark}
		Predicates such as $(\forall x)P(x)$ or $(\exists x)(P(x)$ are disallowed. That is, we do not consider unbounded predicates.
	\end{remark}

	\begin{example}[Words]
		Let $x, y, w, z \in \mathcal{B}$. We consider $(x, y$ to be a word since it is just a finite sequence of letters. Note, a word does not have to \emph{make sense}, in fact, it can be wholly meaningless, as in our example.
		
		We also consider $\bra w, z \ket$ to be a word, by the definition.
	\end{example}
	
	\begin{example}[Terms]
		Let $x, w, z \in \mathcal{B}$. From the last example, $\bra w, z \ket$ is a word, but it is also a term. So we see that there is a bit of structure required to give us a term. Of course, a term does not necessarily have to have meaning.
		
		Using our notation from Remark \ref{R: Evaluation} we see that $(\sin \upharpoonright x) = \sin x$ is a term. Similarly, $\sin \pi$ is a term. Moreover, it is a closed term since it contains only constants.
	\end{example}
	
	\begin{example}[Predicates]
		Let $x, y, z \in \mathcal{B}$ be variables and $A, B, \mathr \in \mathcal{C}$ be constants. We consider $(\forall x \in \mathr)(\exists y \in \mathr)[ (A \upharpoonright x) = y]$ to be a predicate.
		
		We may look at something a bit more familiar, like the statement that a set is dense. $P(B) = (\forall x, y \in B)(\exists z \in B)[x < z < y]$. This another example of a predicate. In this case, $B \in V(\mathr)$ is a free variable.
	\end{example}
	
	\begin{example}[Propositions]
		Considering the density predicate in the above example, we see that letting $B = \mathn$, $B = \mathbb{Q}$, or $B = \mathr$ give propositions that are either true or false.
	\end{example}
	
	With our language in mind, we are prepared to state the Transfer Principle.
	
	\begin{theorem}[Transfer Principle]\index{Transfer Principle}
		Let $P$ be a proposition in the form $P = P(A_1, A_2, ..., A_r)$ for some predicate $P(x_1, x_2, ..., x_r)$ in the language $\mathcal{L}(\mathr)$, some $r \in \mathn$, and some $A_n \in V(\mathr)$. Let ${\star P}$ be the proposition in $\mathcal{L}(\starr)$ obtained by replacing all $A$'s by ${\star A}$'s. That is, ${\star P =: P({\star A_1}, {\star A_2}, ..., {\star A_n})}.$ Then $P$ and ${\star P}$ are equivalent.
	\end{theorem}
	\begin{proof} We refer to M. Davis \cite{mDavis} pp. 24-33 or T. Linstr\o m \cite{tLin} pp. 73-81.
	\end{proof}

\chapter{The Usual Topology of $\mathr$ and Monads}\label{C: Topology}
	We apply our conception of infinitely close to discuss open, closed, bounded, and compact sets. Each shall be characterized with respect to an infinitesimal ball about a point, which we call a monad. We think of monads as a kind of universal infinitesimal interval.

\section{Monads}\label{S: Monads}

	\begin{definition}\index{Monad}
		Let $s \in S$ and $S \subseteq \mathr$.
		\begin{quote}
			\begin{description}
		
				\item[(a)] The set $\mu(s) =: \{s + dx : dx \in \sci(\starr)\}$ is the \textbf{monad at $s$}. 
				\item[(b)] The set $\mu(S) =: \{s + dx : s \in S$, $dx \in \sci(\starr)\}$ is the \textbf{monad on $S$}.\index{Monad!Monad on a Set}
				\item[(c)] The set $\mu_0(s) =: \mu(s) \setminus \{s\}$ is the \textbf{deleted monad at $s$}.\index{Monad!Deleted Monad}
			
			\end{description}
		\end{quote}
	\end{definition}
	Henceforth we shall use the following notation:
	\begin{itemize}
	
		\item The standard $\frac{1}{n}$-ball centered at $s \in \mathr$ is $\nball =: \{ x \in \mathr : | x - s | < \frac{1}{n} \}$.
		\item The standard deleted $\frac{1}{n}$-ball centered at $s$ is $B_0(s, \frac{1}{n}) =: \{ x \in \mathr : 0 < | x - s | < \frac{1}{n} \}$.
		\item The non-standard $\frac{1}{n}$-ball centered at $s$ is $\starball =: \{ x \in \starr : | x - s | < \frac{1}{n} \}$.
		\item The non-standard deleted $\frac{1}{n}$-ball centered at $s$ is ${\star B}_0(s, \frac{1}{n}) =: \{ x \in \starr : 0 < | x - s | < \frac{1}{n} \}$.
		
	\end{itemize}
		
	\begin{lemma}[Characterization of Monad]\label{L: CM}
		Let $s \in \mathr$. Then, $\mu(s) = \cap_{n \in \mathn} \starball$. Similarly, $\mu_0(s) = \cap_{n \in \mathn} {\star B}_0 (s, \frac{1}{n})$.
	\end{lemma}
	\begin{proof}
		Let $s \in \mathr$. We have $x \in \mu(s)$ \iff there exists $dx \in \sci(\starr)$ such that $x =s+ dx$ \iff $x -s\approx 0$ \iff $| x - s | < \frac{1}{n}$ for all $n \in \mathn$ \iff $x \in \cap_{n \in \mathn} \starball$.
	\end{proof}
		
	\begin{lemma}[Balloon Principle]\label{L: BP}
		With $s \in \mathr$ we have $\star B(s, \frac{1}{\nu}) \subset \mu(s) \subset \star B(s, \frac{1}{n})$ for all $n \in \mathn$ and for all $\nu \in \starn \setminus \mathn$.
	\end{lemma}
	\begin{proof}
		From Lemma \ref{L: CM} we have that $\mu(s) \subset \star B(s, \frac{1}{n})$. For the other inclusion observe that $x \in \star B(s, \frac{1}{\nu})$ implies $x \approx s$, which implies $x \in \mu(s)$.
	\end{proof}
		
	The reader may find it interesting that in the statement of the Balloon Principle, both $\nball$ and $B(s, \frac{1}{\nu})$ are internal sets whereas $\mu(s)$ is external. Indeed, an allusion to this is the fact that $\mu(0) = \sci(\starr)$, making $\mu(0)$ external since $\sci(\starr)$ is.
		
\section{Interior Points and Open Sets}\label{S: Interior and Open Sets}
	For $S \subseteq \mathr$ and $s \in \mathr$ recall the following standard definitions:
	\begin{itemize}
	
		\item The point $s \in \mathr$ is an \emph{interior point} of $S$ if there exists $n \in \mathn$ such that $\nball \subseteq S$. We denote the set of interior points of S by $int(S)$.\index{Interior Point}
		\item The set $S$ is \emph{open} \iff $S = int(S)$.\index{Open Set}
	
	\end{itemize}
	
	\begin{theorem}[Characterization of Interior Points and Open Sets]\label{T: IP}
		Let $s \in S$ and $S \subseteq \mathr$. Then $s \in int(S)$ \iff $\mu(s) \subseteq \stars$. Consequently, the following are equivalent:
		\begin{quote}
			\begin{description}
			
				\item[(a)] $S$ is open.
				\item[(b)] $(\forall s \in S)[\mu(s) \subseteq \stars]$.
				\item[(c)] $\mu(S) \subseteq \stars$.
			
			\end{description}
		\end{quote}
	\end{theorem}
	\begin{proof}
		We begin by proving the statement for interior points of $S$.
		\begin{description}
		
			\item[($\Rightarrow$)] Let $s$ be an interior point of $S$. Then there exists $n \in \mathn$ such that $\nball \subseteq S$ which implies there exists $n \in \mathn$ such that $\starball \subseteq \stars$ (by \textbf{(e)} of Theorem \ref{T: Boolean}). Therefore, $\mu(s) \subseteq \stars$ since $\mu(s) \subset \starball$ by Lemma \ref{L: BP}.
			\item[($\Leftarrow$)] Suppose $\mu(s) \subseteq \stars$. Then ${^*B(s, \frac{1}{\nu})} \subset \stars$ for all $\nu \in \starn \setminus \mathn$ by Lemma \ref{L: BP}. Trivially, there exists $\nu_0 \in \starn$ such that ${^*B(s, \frac{1}{\nu_0})} \subset \stars$. By the Transfer Principle there exists $\nu \in \mathn$ such that $B(s, \frac{1}{\nu}) \subset S$, which is to say that $s \in int(S)$.
		
		\end{description}
		
		The equivalence of \textbf{(a)} and \textbf{(b)} is given by noting $S$ open \iff $S = int(S)$ and using the statement just proved. So, $s \in S$ \iff $s \in int(S)$ \iff $\mu(s) \subseteq \stars$. As $s$ was arbitrary, we have that for each $s \in S$, $\mu(s) \subseteq \stars$.
		
		The equivalence of \textbf{(b)} and \textbf{(c)} is clear by the definition of $\mu(S)$.
	\end{proof}
		
	\begin{remark}[Reduction of Quantifiers]\index{Reduction of Quantifiers}
		Notice that the standard formulation of open set requires two quantifiers: $(\forall s \in S)(\exists n \in \mathn)[\nball \subseteq S]$. On the other hand, the non-standard formulation requires \emph{none}: $\mu(S) \subseteq \stars$.
	\end{remark}

\section{Cluster Points, Adherent Points, and Closed Sets}\label{S: Cluster and Closed}
	For $S \subseteq \mathr$ and $s \in \mathr$ recall the following standard definitions:
	\begin{itemize}
	
		\item The point $s$ is a \emph{cluster point of $S$} if for all $\varep \in \rplus$ the set $\Sep = B_0(s, \varep) \cap S$ is non-empty. The set of all cluster points is denoted $S'$.\index{Cluster Point}
		\item The point $s$ is \emph{adherent to $S$} if $\nball \cap S \not= \varnothing$ for all $n \in \mathn$. The set of all points adherent to $S$ is called the \emph{closure} and is denoted $\overline{S}$.\index{Adherent Point}
		\item The set $S$ is \emph{closed} if $S = \overline{S}$.\index{Closed Set}
	
	\end{itemize}
	
	The following characterization will become useful for us when we discuss the non-standard characterization of limits.

	\begin{theorem}[Characterization of Cluster Points]\label{T: AP}
		Let $s \in \mathr$ and $S \subseteq \mathr$. Then the following statements are equivalent:
		\begin{quote}
			\begin{description}
			
				\item[(a)] $s$ is a cluster point of $S$.
				\item[(b)] $s + dx \in \stars$ for some non-zero infinitesimal $dx$.
				\item[(c)] There exists $x \in \stars$ such that $x \approx s$ and $x \not= s$.
				\item[(d)] $\stars \cap \mu_0(s) \not= \varnothing$.
				
			\end{description}
		\end{quote}
		Consequently, $S' = \{ s \in \mathr : \stars \cap \mu_0(s) \not= \varnothing \}$.
	\end{theorem}
	\begin{proof}
		We first prove the the equivalence of \textbf{(a)} and \textbf{(b)}.
		\begin{description}
	
			\item[($\Rightarrow$)] Let $s$ be a cluster point of $S$. Then for all $\varep \in \rplus$ the set $\Sep = B_0(s, \varep) \cap S$ is non-empty, and by the Axiom of Choice there exists a net $(\xep) \in S_{\varep}^{\rplus}$ such that $\xep \in \Sep$ for all $\varep \in \rplus$. Equivalently, $\xep \in B_0(s, \varep) \cap S$ for all $\varep \in \rplus$ so we have $\{ \varep \in \rplus : \xep \in B_0(s, \varep) \cap S\} = \rplus \in \scu$. Note, $\{ \varep \in \rplus : \xep \in B_0(s, \varep) \cap S\} = \{ \varep \in \rplus : \xep \in S$ and $0 < | \xep - s | < \varep\}$ so we conclude that $\bra \xep \ket \in \stars$, and by letting $dx =: \bra \xep -s\ket = \bra \xep \ket - s$, we conclude that $dx$ is a non-zero infinitesimal.
			\item[($\Leftarrow$)] Let $s + dx \in \stars$ for some non-zero infinitesimal $dx$. Looking at representatives, $dx = \bra \delta_\varep \ket$ for some $(\delta_\varep) \in \mathr^{\rplus}$. For each $n \in \mathn$ define $S_n = \{ \varep \in \rplus : 0 < | \delta_\varep | < \frac{1}{n}$ and $s + \delta_\varep \in S$ \}. Then $S_n \in \scu$ by our assumptions and is therefore non-empty for all $n \in \mathn$. If we let $\xep =:s+ \delta_\varep$, then $\xep \in S$ and $0 < | \xep -s| < \frac{1}{n}$ for all $n \in \mathn$, which means that $s$ is a cluster point of $S$.
	
		\end{description}
	
	The equivalence of \textbf{(b)} and \textbf{(c)} is immediate by taking $x =: s+ dx$, and the equivalence of \textbf{(c)} and \textbf{(d)} is clear from the definition of $\mu_0(s)$.
	\end{proof}
	
	\begin{remark}[Reduction of Quantifiers]\index{Reduction of Quantifiers}
		Notice that the standard formulation of a cluster point requires one quantifier: $(\forall \varep \in \rplus)[B_0 (s, \varep) \cap S \not= \varnothing]$. On the other hand, the non-standard formulation requires \emph{none}: $\stars \cap \mu_0(s) \not= \varnothing$.
	\end{remark}
	
	We use the work on cluster points to discuss adherent points, and then of closed sets.

	\begin{theorem}[Characterization of Adherent Points]\label{T: NSAP}
		Let $s \in \mathr$ and $S \subseteq	\mathr$. Then $s \in \overline{S}$ \iff $\stars \cap \mu(s) \not= \varnothing$. Consequently, $\overline{S} = \{ s \in \mathr : \stars \cap \mu(S) \not= \varnothing \}$.
	\end{theorem}
	\begin{proof} 
	\begin{description}
		
		\item[($\Rightarrow$)] In the case that $s$ is an isolated point of $S$ the result is trivially true. In the case that $s$ is not an isolated point it is then a non-trivial adherent point, or a cluster point of $S$. Then by Theorem \ref{T: AP} part \textbf{(b)} we have $\stars \cap \mu(s) \not= \varnothing$. 
		\item[($\Leftarrow$)] Assume $\stars \cap \mu(s) \not= \varnothing$. Then for some non-zero infinitesimal $dx$ we have $s + dx \in \stars$ and by Theorem \ref{T: AP} we have that $s$ is a cluster point. As any cluster point is an adherent point we are done. 
		
	\end{description}
	\end{proof}
	
	\begin{remark}[Reduction of Quantifiers]\index{Reduction of Quantifiers}
		Notice that the standard formulation of a adherent point requires one quantifier: $(\forall \varep \in \rplus)[B (s, \varep) \cap S \not= \varnothing]$. On the other hand, the non-standard formulation requires \emph{none}: $\stars \cap \mu(s) \not= \varnothing$.
	\end{remark}
	
	Before characterizing the closed sets we must extend our definition of the $\st$-mapping a bit further from Definition \ref{D: SPM}.

	\begin{definition}\index{Standard Part Mapping!on a Set}
		Let $A \subseteq \starr$. We define $\st(A) =: \{ \st(x) : x \in A \cap \scf(\starr) \}$.
	\end{definition}
	
	Notice that $\st(\mu(S)) = S$ for any $S \subseteq \mathr$ by the definition of $\mu(S)$.

	\begin{theorem}[Characterization of Closed Sets]\label{T: NSCS}
		Let $s \in \mathr$ and $S \subseteq \mathr$. Then the following are equivalent:
		\begin{quote}
			\begin{description}
			
				\item[(a)] $S$ is closed.
				\item[(b)] For each $s \in S$, if $\stars \cap \mu(s) \not= \varnothing$, then $s \in S$.
				\item[(c)] $S = \{ s \in \mathr : \stars \cap \mu(s) \not= \varnothing \}$.
				\item[(d)] $S = \st(\stars)$.
				
			\end{description}
		\end{quote}
	\end{theorem}
	\begin{proof} First, we prove the equivalence of \textbf{(a)} and \textbf{(b)}.
	\begin{description}
	
		\item[($\Rightarrow$)] The set $S$ is closed \iff $(\mathr \setminus S)$ is open \iff for all $r \in (\mathr \setminus S)$ we have $\mu(r) \subseteq \star(\mathr \setminus S)$. Let $s \in \mathr$ be such that $\stars \cap \mu(s) \not= \varnothing$. Either $s \in S$ or $s \in (\mathr \setminus S)$. If $s \in (\mathr \setminus S)$ then $\mu(s) \subseteq \star(\mathr \setminus S)$ which implies $\stars \cap \mu(s) = \varnothing$, a contradiction since ${\star(\mathr \setminus S)} = {\star \mathr} \setminus {\star S}$. Therefore $s \in S$.
		\item[($\Leftarrow$)] For every $s \in \mathr$ we have $\stars \cap \mu(s) \not= \varnothing$ implies $s \in S$ by assumption. Suppose to the contrary that $S$ is not closed. Then there exists $r \in \overline{S} \setminus S$. Thus $B(r; \frac{1}{n}) \cap S \not= \varnothing$ for all $n \in \mathn$. Thus ${\star B}(r; \frac{1}{n}) \cap \stars \not= \varnothing$ for all $n \in \mathn$ since $B(r; \frac{1}{n}) \cap S \subseteq {\star B}(r; \frac{1}{n}) \cap \stars$. Further, note that the latter is internal so by the Saturation Principle $\cap_{n \in \mathn} ( \star B(r ; \frac{1}{n}) \cap \stars ) = \cap_{n \in \mathn} \star B(r ; \frac{1}{n}) \cap \stars = \mu(r) \cap \stars \not= \varnothing$. Therefore $r \in S$, a contradiction.
	
	\end{description}
	
	The equivalence of \textbf{(b)} and \textbf{(c)} follows since $S$ closed \iff $S = \overline{S}$ and $\overline{S} = \{ s \in \mathr : \stars \cap \mu(s) \not= \varnothing \}$ from Theorem \ref{T: NSAP}.
	Last, we prove the equivalence of \textbf{(c)} and \textbf{(d)}.
	\begin{description}
	
		\item[($\Rightarrow$)] With $s \in S$ we have $\st(s) = s$. Therefore $s \in \stars \cap \scf(\starr)$ giving $s \in \st(\stars)$. Conversely, let $s \in \st(\stars)$. Then $s = \st(x)$ for some finite $x \in \stars$. Since $\stars \cap \mu(s) \not= \varnothing$ we have $s \in S$ by assumption.
		\item[($\Leftarrow$)] We have $S \subseteq \{ s \in \mathr : \stars \cap \mu(s) \not= \varnothing \}$ trivially since $S \subseteq \stars$. Conversely, suppose $\stars \cap \mu(s)$ for some $s \in \mathr$. We have $s = \st(x)$ for some $x \in \stars \cap \mu(s)$ which implies $x \in \stars \cap \scf(\starr)$ and therefore $s \in S$ (since $S = \st(\stars)$). 
	
	\end{description}
	\end{proof}
	
	\begin{remark}[Reduction of Quantifiers]\index{Reduction of Quantifiers}
		Notice that the standard formulation of a closed set requires two quantifiers: $(\forall s \in S)(\forall \varep \in \rplus)[B(s, \varep) \cap S \not= \varnothing]$. On the other hand, the non-standard formulation requires \emph{none}: $S = \st(\stars)$.
	\end{remark}

\section{Bounded Sets}\label{S: Bounded Sets}
	Recall that a standard set $A \subset \mathr$ is \textbf{bounded} if there exists $M \in \rplus$ such that for all $s \in S$ we have $|s| \leq M$.\index{Bounded Sets}
	
	\begin{theorem}[Characterization of Bounded Sets]\label{T: SBS}
		Let $S \subseteq \mathr$. Then $S$ is bounded \iff $\stars \subseteq \scf(\starr)$. Similarly, $S$ is bounded from above (below) \iff $\stars$ does not contain positive (negative) infinitely large numbers.
	\end{theorem}
	\begin{proof}
	\begin{description}
	
		\item[($\Rightarrow$)] Suppose $S$ is bounded, and let $M \in \rplus$ be a bound for $S$. Then, $(\forall s \in S)(|s| \leq M)$ \iff (by Transfer Principle) $(\forall s \in \stars)(|s| \leq M)$ which implies $\stars \subseteq \scf(\starr).$
		\item[($\Leftarrow$)] Suppose $\stars \subseteq \scf(\starr)$. Then $(\exists M \in \starr_+)(\forall s \in \stars)(|s| \leq M)$. By the Transfer Principle we have $(\exists M \in \rplus)(\forall s \in S)(|s| \leq M)$. That is, $S$ is bounded. 
	
	\end{description}
	\end{proof}
	
	\begin{remark}[Reduction of Quantifiers]\index{Reduction of Quantifiers}
		Notice that the standard formulation of a closed set requires two quantifiers: $(\exists M \in \rplus)(\forall s \in S)[|s| \leq M]$. On the other hand, the non-standard formulation requires \emph{none}: $\stars \subseteq \scf(\starr)$.
	\end{remark}

\section{Compact Sets}\label{S: Compact Sets}
	Recall that a subset $S \subseteq \mathr$ is compact \iff it is closed and bounded. We shall use this characterization rather than open covers for the sake of simplicity.\index{Compact Sets}
	
	\begin{theorem}[Characterization of Compact Sets]\label{T: NSCoS}
		Let $S \subseteq \mathr$. Then the following are equivalent:
		\begin{quote}
			\begin{description}
			
				\item[(a)] $S$ is compact.
				\item[(b)] $\stars \subseteq \mu(S)$.
				\item[(c)] $\stars \subseteq \bigcup_{s \in S} \mu(s)$.
				
			\end{description}
		\end{quote}
	\end{theorem}
	\begin{proof}
		First, we prove the equivalence of \textbf{(a)} and \textbf{(b)}.
		\begin{description}
		
			\item[($\Rightarrow$)] $S$ compact \iff $S$ is closed and bounded \iff $S = \st(\stars)$ and $\stars \subseteq \scf(\starr)$ (by part \textbf{(d)} of Theorem \ref{T: NSCS} and Theorem \ref{T: SBS}, respectively). Let $x \in \stars$, then $x \in \scf(\starr)$ and by Theorem \ref{T: Representation} we have $x = \st(x) + dx$ for some $dx \approx 0$. Note, $\st(x) \in \st(\stars) = S$ therefore $x \in \mu(S)$.
			\item[($\Leftarrow$)] Suppose $\stars \subseteq \mu(S)$. As $\mu(S) \subseteq \scf(\starr)$ we have $S$ bounded by Theorem \ref{T: SBS}. We must show that $S = \st(\stars)$. Trivially we have $S \subseteq \st(\stars)$. For the reverse containment we have $s \in \st(\stars)$ \iff $s = \st(x)$ for some $x \in \stars \cap \scf(\mathr)$. Therefore $x \in \mu(S)$ (from our assumption), giving $s = \st(x) \in S$. Therefore $S$ is closed and bounded.
		
		\end{description}
		The equivalence of \textbf{(b)} and \textbf{(c)} is given by noting that $\mu(S) = \bigcup_{s \in S} \mu(s)$. 
	\end{proof}
	
	\begin{remark}[Reduction of Quantifiers]\index{Reduction of Quantifiers}
		Notice that the standard formulation of a compact set requires two quantifiers: Every open cover has a finite subcover. On the other hand, the non-standard formulation requires \emph{none}: $\stars \subseteq \bigcup_{s \in S} \mu(s)$.
	\end{remark}
	
\chapter{Topics in Real Analysis in a Non-Standard Setting}\label{C: Analysis}

We give non-standard characterizations for such standard analytic concepts as: sequences, limits, continuity, uniform continuity, derivatives, sequences of functions, and uniform convergence. We emphasis with each characterization the reduction of quantifiers.
		
\section{Limits}\label{S: Limits}
	In our discussion of the $\st$-mapping, the reader may have noticed its resemblance to taking a classical limit. This observation is quite valid and we devote this section to proving it.
		
	Let $r \in \mathr$ be a cluster point of $X \subseteq \mathr$ (Theorem \ref{T: AP}). Let $f: X \to \mathc$ be a function. Recall the following standard definitions:
	
	\begin{itemize}
		
		\item $\lim_{x \to r} f(x) = L$ if, by definition, for each $\varep \in \rplus$ there exists $\delta \in \rplus$ such that for all $x \in X$, if $0 < | x-r | < \delta$ then $|f(x) - L| < \varep$.
		\item Phrased in countable variables, $\lim_{x \to r} f(x) = L$ if, by definition, for each $m \in \mathn$ there exists $n \in \mathn$ such that for all $x \in X$, if $0 < | x-r | < \frac{1}{n}$, then $|f(x) - L| < \frac{1}{m}$.
		\item $\lim_{x \to r} f(x) \not= L$ \iff there exists $m \in \mathn$ such that for all $\delta \in \mathr$ there exists $x \in X$ so that $0 < |x -r | < \delta$ and $|f(x) - L| \geq \frac{1}{m}$.

	\end{itemize}
	
	\begin{theorem}[Limits]\label{T: NSL}\index{Limit}
		Let $f : X \to \mathc$ where $X \subseteq \mathr$. Suppose $r$ is a cluster point of $X$ and $L \in \mathc$ (Theorem \ref{T: AP}). Then the following statements are equivalent:
		\begin{quote}
			\begin{description}
				
				\item[(a)] $\lim_{x \to r} f(x) = L$.
				\item[(b)] $\starf(r + dx) \approx L$ for all non-zero infinitesimals $dx$, such that $r + dx \in \starx$.
				\item[(c)] $\starf(x) \approx L$ for all $x \in \starx$ such that $x \approx r$, $x \not= r$.
				\item[(d)] $\st[\starf(r + dx)] = L$ for all non-zero infinitesimals $dx$, such that $r + dx \in \starx$.
			
			\end{description}
		\end{quote}
	\end{theorem}
	\begin{proof}
		We prove the equivalence of \textbf{(a)} and \textbf{(b)} in detail.
		
		From Theorem \ref{T: AP} we know that $r$ a cluster point for $X \subseteq \mathr$ \iff $r + dx \in \starx$ for some non-zero infinitesimal $dx$. As in the proof of Theorem \ref{T: AP} we have $dx = \bra \delta_\varep \ket$ for some $(\delta_\varep) \in \mathr^{\rplus}$ and $A_n =: \{ \varep \in \rplus : 0 < | \delta_\varep | < \frac{1}{n}$ and $r + \delta_\varep \in X \} \in \scu$ for all $n \in \mathn$.
		\begin{description}
		
			\item[($\Rightarrow$)] We use the countable formulation of the definition of the limit. Fixing $m \in \mathn$, there exists a $n \in \mathn$ such that for all $x \in X$, if $0 < | x-r | < \frac{1}{n}$ then $|f(x) - L| < \frac{1}{m}$. Define $B_m =: \{ \varep \in \rplus : |f(r + \delta_\varep) - L | < \frac{1}{m}\}$. Notice that $A_n \subseteq B_m$. Indeed, let $\varep \in A_n$ so that $0 < |\delta_\varep | < \frac{1}{n}$ for all $n \in \mathn$ and $r + \delta_\varep \in X$. Let $x =: r + \delta_\varep$, then $| f(r + \delta_\varep) - L < \frac{1}{m}$ so that $\varep \in B_m$. As $A_n \in \scu$, by the properties of ultrafilters, $B_m \in \scu$. So $|\starf(r + dx) - L | < \frac{1}{m}$, and since $m \in \mathn$ was arbitrary, $\starf(r + dx) - L \approx 0$ as desired.
			\item[($\Leftarrow$)] Assume that $\starf(r + dx) \approx L$ for all non-zero infinitesimals $dx$ such that $r + dx \in \starx$. We shall use the hybrid (continuous/countable) negation of the definition of the limit. Suppose to the contrary that $\lim_{x \to c} f(x) \not= L$. So there exists $m \in \mathn$ such that for all $\delta \in \rplus$ the set $\Xdlt =: \{ x \in X : 0 < |x - r | < \delta$ and $|f(x) - L| \geq \frac{1}{m}\}$ is non-empty for all $\delta \in \rplus$. By the Axiom of Choice there exists $(\xdlt) \in \mathr^{\rplus}$ such that $\xdlt \in \Xdlt$ for all $\delta \in \rplus$. Define $dx =: \bra \xdlt - r\ket$. As $\xdlt \in \Xdlt$ for all $\delta \in \rplus$ we have $r + dx \in \starx$. Also, we have $0 < |dx| < \bra \delta \ket$ and $|\starf(r + dx) - L| \geq \frac{1}{m}$. Therefore, $dx$ is a non-zero infinitesimal and $\starf(r + dx) - L \not\approx 0$, a contradiction.
		
		\end{description}
		Note the equivalence of \textbf{(b)} and \textbf{(c)} is immediate by letting $x =: r + dx$, and we obtain the equivalence of \textbf{(c)} and \textbf{(d)} by noting $L \in \mathc$ and applying Theorem \ref{T: Representation}. 
	\end{proof}
		
	\begin{remark}[Reduction of Quantifiers]\index{Reduction of Quantifiers}
		A particularly nice feature of the non-standard characterization of limits is that the number of quantifiers is reduced as compared to the standard characterization. Indeed, observe the formalization of the standard definition for the limit of $f: X \to \mathr$ at the point $c \in \mathr$ with limit $L \in \mathr$:
		\begin{equation}\label{E: SLIM}
			(\forall \varep \in \rplus)(\exists \delta \in \rplus)(\forall x \in X)[0 < | x - c | < \delta \Rightarrow | f(x) - L | < \varep]
		\end{equation}
		Now observe the formalization of the non-standard characterization from Theorem \ref{T: NSL}:
		\begin{equation}\label{E: NSLIM}
			(\forall x \in \starx)[x \approx r \Rightarrow f(x) \approx L]
		\end{equation}
		
		Notice that whereas the former has three quantifiers, the latter only has one. Moreover, observe that the non-standard characterization is intuitively what we think of as a limit, but we have made rigorous the idea of infinitely close!
	\end{remark}
	
	We now present some examples to illustrate our characterization of the limit.
		
	\begin{examples}\label{E: Lim}$ $
		\begin{quote}
		\begin{description}
	
			\item[(i)] $\lim_{x \to 1} \frac{x}{1+x} = \frac{1}{2}$. Indeed, let $dx$ be a non-zero infinitesimal, then $1 + dx \in \starr \setminus \{-1\}$. We have $\st[\starf(1 + dx)] = \st[\frac{1 + dx}{2 + dx}] = \frac{\st[1 + dx]}{\st[2+dx]} = \frac{1}{2}$. Hence, by part \textbf{(d)} of Theorem \ref{T: NSL} the limit is $\frac{1}{2}$.
			\item[(ii)] If $x \approx 0$, then $\star\sin x \approx 0$. Indeed by Theorem \ref{T: NSL} $\st[\star\sin x] = \lim_{x \to 0} \sin x = 0$. Therefore, $\star\sin x \approx 0$.
	
		\end{description}
		\end{quote}
	\end{examples}
	
\section{Limits at Infinity}\label{S: Limits at Infinity}
	To facilitate a discussion on the limit of a sequence in the language of non-standard analysis we first discuss the limit of a function as $x$ goes to infinity.

	Let $X \subseteq \mathr$ be unbounded from above, let $f: X \to \mathc$, and suppose $L \in \mathc$. Recall that $\lim_{x \to \infty} f(x) = L$ if (by definition),
	\begin{equation}\label{E: Lim to Infinity}
		(\forall \varep \in \rplus)(\exists K \in \rplus)(\forall x \in X)[x > K \Rightarrow |f(x) - L| < \varep].
	\end{equation}

	In the following, $\starx_+$ denotes the set of positive numbers in $\starx$ and $\scl(\starx_+)$ denotes the set of infinitely large numbers in $\starx$. Notice that $\scl(\starx_+) \not= \varnothing$ by Theorem \ref{T: SBS}. 

	\begin{theorem}[Characterization of Limits at Infinity]\label{T: NSLI}\index{Limit!at Infinity}
		Let $X \subseteq \mathr$ be a set which is unbounded from above. Let $f: X \to \mathc$ and suppose $L \in \mathc$. Then $\lim_{x \to \infty} f(x) = L$ \iff $(\forall x \in \scl(\starx_+))[\starf(x) \approx L]$.
	\end{theorem}
	\begin{proof}$ $
		\begin{description}
		
			\item[($\Rightarrow$)] We assume (\ref{E: Lim to Infinity}). Let $\varep \in \rplus$ be fixed so that there exists $K \in \rplus$ such that $(\forall x \in X)[x > K \Rightarrow |f(x) - L| < \varep]$. Apply the Transfer Principle so that $(\forall x \in \starx)[x > K \Rightarrow |\starf(x) - L| < \varep]$. As $K$ is standard, picking any $x \in \scl(\starx_+)$ gives $x > K$ so that $(\forall x \in \scl(\starx_+))[\starf(x) - L| < \varep]$, which implies $(\forall x \in \scl(\starx_+))[\starf(x) \approx L]$, as desired.
			\item[($\Leftarrow$)] Assume $(\forall x \in \scl(\starx_+)(\starf(x) \approx L)$. Let $\varep \in \rplus$ (arbitrarily fixed). Then $(\forall x \in \scl(\starx_+))[|\starf(x) - L| < \varep]$. Trivially, $(\exists K \in \starr_+)(\forall x \in \starx)[x > K \Rightarrow |\starf(x) - L| < \varep]$ (For instance, pick $K \in \starn \setminus \mathn$). Apply the Transfer Principle so that, $(\exists K \in \rplus)(\forall x \in X)[x > K \Rightarrow |f(x) - L| < \varep]$. As $\varep \in \rplus$ was arbitrary, we have $\lim_{n \to \infty} f(x) = L$.
		
		\end{description}
		We present an alternate proof which uses the Overflow Principle.
		\begin{description}
		
			\item[($\Leftarrow$)] Assume $(\forall x \in \scl(\starx_+))[\starf(x) \approx L]$. This translates to $(\forall \varep \in \rplus)(\forall x \in \scl(\starx_+))[|\starf(x) - L| < \varep]$. Suppose to the contrary that $(\exists \epn \in \rplus)(\forall K \in \rplus)(\exists x \in X)[x > K$ and $|f(x) - L| \geq \epn]$. Fix $\epn \in \rplus$ so that, $$(\forall K \in \rplus)(\exists x \in X)[x > K \textnormal{ and } |f(x) - L| \geq \epn].$$ Apply the Transfer Principle so that
			\begin{equation}\label{E: Nonempty}
				(\forall K \in \starr_+)(\exists x \in \starx)[x > K \textnormal{ and } |\starf(x) - L| \geq \epn].
			\end{equation}
			Consider $A =: \{ x \in \starx : |\starf(x) - L| \geq \epn \}$. From (\ref{E: Nonempty}) we know that $A \not= \varnothing$. Either $A$ contains infinitely large positive numbers (in which case we contradict our given assumption), or $A$ contains arbitrarily large finite numbers. In this case, we apply the Overflow Principle so that $A$ contains at least one infinitely large number, contradicting our given assumption. 
		
		\end{description}
	\end{proof}
	
	\begin{example}
	$\lim_{x \to \infty} \frac{\sin x}{x} = 0$. Let $dx$ be a positive infinitesimal. Notice that we may neither use l'Hospital's rule (the limit in the numerator does not exist), nor may we distribute the limit (for the same reason). We have $$\lim_{x \to \infty} \frac{\sin x}{x} = \lim_{x \to 0^+} \frac{\sin \frac{1}{x}}{\frac{1}{x}} = \lim_{x \to 0^+} x \sin \frac{1}{x} = \st \left[ dx \sin\left(\frac{1}{dx}\right)\right] = \st(dx)\st\left[\sin\left(\frac{1}{dx}\right)\right]= 0$$ as required since $\st(dx) = 0$ and $\st\left(\frac{1}{dx}\right) \in \mathr$ since $\sin\left(\frac{1}{dx}\right)$ is finite (as a number in ${\star[-1, 1]}$.
	\end{example}
	
	\begin{remark}[Reduction of Quantifiers]\index{Reduction of Quantifiers}
		Again, whereas the standard definition of a limit at infinity contains three quantifiers, $(\forall \varep \in \rplus)(\exists K \in \rplus)(\forall x \in X)[x > K \Rightarrow |f(x) - L| < \varep],$ the non-standard characterization in Theorem \ref{T: NSLI}, $(\forall x \in \scl(\starx_+))[\starf(x) \approx L]$, has a \emph{single} quantifier.
	\end{remark}
	
	The characterization of the limit of a sequence follows from the characterization in Theorem \ref{T: NSLI}. We take $X = \mathn$ and note that $\scl(\starn) = \starn \setminus \mathn$.
		
	\begin{corollary}[Characterization of Limits of Sequences]\label{C: NSLS}\index{Limit!of a Sequence}
		Let $(a_n) \in \mathc^\mathn$ be a sequence, and suppose $L \in \mathc$. Then $\lim_{n \to \infty} a_n = L$ \iff $(\forall n \in \starn \setminus \mathn)[{^*a_n} \approx L].$
	\end{corollary}
	\begin{proof}
		Take $X = \mathn$ and $f(n) = a_n$. Then the proof follows directly from Theorem \ref{T: NSLI}. 
	\end{proof}
		
	\begin{examples}$ $
		\begin{quote}
		\begin{description}
		
			\item[(i)] $\lim_{n \to \infty} \frac{\sqrt{n + 1}}{n} = 0$. Indeed, let $\nu \in \starn \setminus \mathn$. Then $\frac{\nu + 1}{\nu} = \frac{\nu}{\nu}\sqrt{1 + \frac{1}{\nu}} = \frac{1}{\sqrt{\nu}}\sqrt{1 + \frac{1}{\nu}} \approx 0$, since $\frac{1}{\sqrt{\nu}} \approx 0$ and $\sqrt{1 + \frac{1}{\nu}} \approx 1$.
			\item[(ii)] $\lim_{n \to \infty} \frac{n + 5}{n + 3} = 1$. Indeed, let $\nu \in \starn \setminus \mathn$. Then $\frac{\nu + 5}{\nu + 3} = \frac{ 1 + \frac{5}{\nu}}{1 + \frac{3}{\nu}} \approx 1$.
		\end{description}
		\end{quote}
	\end{examples}
		
\section{Continuity}\label{S: Continuity}
	We characterize ordinary continuity on (both at a point and on a set) and then discuss the more difficult characterization of uniform continuity.
	
	Let $r \in X$ and $X \subseteq \mathr$. Recall the following standard definitions:
	
	\begin{itemize}
	
		\item $f: X \to \mathc$ is \emph{continuous at the point $r$} if, by definition, for all $\varep \in \rplus$, there exists $\delta \in \rplus$ such that for all $x \in X$, if $| x - r |< \delta$, then $|f(x) - f(r)| < \varep$.
		\item Consequently, $f: X \to \mathc$ is \emph{continuous on the set $X$} if, by definition, for all $r \in X$ and $\varep \in \rplus$, there exists $\delta \in \rplus$ such that for all $x \in X$, if $|x - r|<\delta$, then $|f(x) - f(r)|< \varep$.
	
	\end{itemize}

	Note that countable formulations may be made as before, and we shall freely use them.
	
	\begin{theorem}[Continuity]\label{T: NSC}\index{Continuity!at a Point}
		Let $X \subseteq \mathr$, $r \in X$, and $f: X \to \mathc$. The following statements are equivalent:
		\begin{quote}
			\begin{description}
			
				\item[(a)] $f$ is continuous at the point $r$.
				\item[(b)] $\starf(r + dx) \approx \starf(r)$ for all infinitesimals $dx$ with $r + dx \in \starx$.
				\item[(c)] $\starf(x) \approx \starf(r)$ for all $x \in \starx$ with $x \approx r$.
			
			\end{description}
		\end{quote}
	\end{theorem}
	\begin{proof}
		Since $r \in X$ by assumption, we have $\starf(r) = f(r)$. The equivalence of \textbf{(a)} and \textbf{(b)} follows by letting $f(r) = L$ and applying Theorem \ref{T: NSL}. The equivalence of \textbf{(b)} and \textbf{(c)} is immediate by letting $x =: r + dx$. 
	\end{proof}
		
	\begin{corollary}\index{Continuity!on a Set}
		Let $X \subseteq \mathr$ and $f: X \to \mathc$. Then $f$ is continuous on the set $X$ \iff $\starf(x) \approx f(r)$ for all $r \in X$ and $x \in \starx$ such that $x \approx r$.
	\end{corollary}
	
	\begin{remark}[Reduction of Quantifiers]\index{Reduction of Quantifiers}
		We formalize both standard and non-standard characterizations of continuity with our focus on comparing the quantifiers. Observe that $f: X \to \mathr$ is continuous on the set $X \subset \mathr$ if:
		\begin{equation}\label{E: CSET}
			(\forall r \in X)(\forall \varep \in \rplus)(\exists \delta \in \rplus)(\forall x \in X)[ |x - r| < \delta \Rightarrow |f(x) - f(r)| < \varep].
		\end{equation}
		Observe the non-standard characterization:
		\begin{equation}\label{E: NSCSET}
				(\forall r \in X)(\forall x \in \starx)[x \approx r \Rightarrow \starf(x) \approx f(r)].
			\end{equation}
			
		Again we notice that while the former has four non-commuting quantifiers, the latter has two \emph{commuting} quantifiers.
	\end{remark}
	
	We again turn to some examples to demonstrate our characterization.
	
	\begin{examples}\label{E: Cont}$ $
		\begin{quote}
		\begin{description}
		
			\item[(i)] The function $f: \mathr \to \mathr$ defined by $f(x) = x$ is continuous on $\mathr$. To show this, let $r \in \mathr$ be fixed and arbitrary. Suppose $x \in \starr$ such that $x \approx r$. Then $\starf(x) = x \approx r = \starf(r)$, and by part \textbf{(c)} of Theorem \ref{T: NSC} $f(x) = x$ is continuous at $r \in \rplus$. Since $r$ was arbitrary, $f(x) = x$ is continuous on $\mathr$.
			\item[(ii)] The function $f: \rplus \to \mathr$ defined by $f(x) = \frac{1}{x}$ is continuous on $\rplus$. Let $r \in \mathr$ be fixed and arbitrary. Let $dx$ be a non-zero infinitesimal so that $r + dx \in \starr_+$. Then $\st[\starf(r + dx) - \starf(r)] = \st[\frac{1}{r + dx} - \frac{1}{r}] = \st[\frac{dx}{r^2 + rdx}] = \frac{\st[dx]}{\st[r^2 + rdx]} = \frac{0}{r^2} = 0$. Therefore, $\starf(r + dx) \approx \starf(r)$ so by part \textbf{(b)} of Theorem \ref{T: NSC} $f(x) = \frac{1}{x}$ is continuous at $r \in \rplus$. As $r$ was arbitrary, $f(x) = \frac{1}{x}$ is continuous on $\rplus$.
		
		\end{description}
		\end{quote}
	\end{examples}
	
\section{Uniform Continuity}\label{S: Uniform Continuity}

	As was shown, the continuity of a function (both at a point and on a set) followed directly by substitution from the work done on limits. Unfortunately we are not so lucky with respect to \textbf{uniform continuity}. As with ordinary continuity, we recall the formalized standard definition of uniform continuity. The function $f: X \to \mathc$ is uniformly continuous on $X \subseteq \mathr$ if, by definition, for all $\varep \in \rplus$ there exists $\delta \in \rplus$ such that for all $x, r \in X$, $|x - r|< \delta$ implies $|f(x) - f(r)| < \varep$.
		
	\begin{theorem}[Uniform Continuity]\label{T: NSUC}\index{Uniform Continuity}
		Let $X \subseteq \mathr$. The function $f: X \to \mathc$ is uniformly continuous on $X$ \iff $\starf(x) \approx \starf(r)$ for all $x, r \in \starx$ such that $x \approx r$.
	\end{theorem}
	\begin{proof}$ $
		\begin{description}
		
			\item[($\Rightarrow$)] Suppose $f$ is uniformly continuous on $X \subseteq \mathr$. Let $m \in \mathn$, then there exists $n \in \mathn$ such that for all $u, v \in X$, $|u - v| < \frac{1}{n}$ implies $|f(u) - f(v)| < \frac{1}{m}$. Suppose $x = \bra \xep \ket$ and $r = \bra \rep \ket$ are in $\starx$ such that $x \approx r$. We must show that $\starf(x) \approx \starf(r)$. Indeed, $A_n =: \{ \varep \in \rplus : |\xep - \rep| < \frac{1}{n}\} \in \scu$. As $f$ is uniformly continuous, $A_n \subset B_m =: \{ \varep \in \rplus : |f(\xep) - f(\rep)| < \frac{1}{m}\}$ and so $B_m \in \scu$. Therefore, $|\starf(x) - \starf(r)| < \frac{1}{m}$ for all $m \in \mathn$ which implies $\starf(x) \approx \starf(r)$.
			\item[($\Leftarrow$)] Assume for all $x, r \in \starx$, $x \approx r$ implies $\starf(x) \approx \starf(r)$. Suppose to the contrary that $f$ is not uniformly continuous on $X$. That is, there exists $\epn \in \rplus$ such that for all $\varep \in \rplus$ there exists $u, v \in X$ such that $|u - v| < \varep$ and $|f(u) - f(v)| \geq \epn$. For each $\varep \in \rplus$ the set $\Aep =: \{ (u,v) \in X \times X : |u - v| < \varep$ and $|f(u) - f(v)| \geq \epn\}$ is non-empty. By the Axiom of Choice there exist nets $(\xep)$ and $(\rep)$ such that $(\xep, \rep) \in \Aep$ for all $\varep \in \rplus$. Therefore, $x =: \bra \xep \ket$ and $r =: \bra \rep \ket$ are each in $\starx$. So we have $|x - r| < \bra \varep \ket$, giving $x \approx r$ (since $\bra \varep \ket \approx 0$) and $|\starf(x) - \starf(r)| \geq \epn$. Hence, $\starf(x) \not\approx \starf(r)$, a contradiction. 
		
		\end{description}
	\end{proof}
			
	\begin{remark}[Reduction of Quantifiers]\index{Reduction of Quantifiers}
		Once again, we formalize both standard and non-standard characterizations of uniform continuity in order to compare the quantifiers. Recall that $f: X \to \mathr$ is uniformly continuous on $X \subseteq \mathr$ if:
		\begin{equation}\label{E: UC}
			(\forall \varep \in \rplus)(\exists \delta \in \rplus)(\forall x,r \in X)[ |x - r| < \delta \Rightarrow |f(x) - f(r)| < \varep].
		\end{equation}
		Likewise, the non-standard characterization from Theorem \ref{T: NSUC} states:
		\begin{equation}\label{E: NSUC}
			(\forall r, x \in \starx)[x \approx r \Rightarrow \starf(x) \approx \starf(r)].
		\end{equation}
		
		Yet again, the former has four non-commuting quantifiers while the latter has has two \emph{commuting} quantifiers. 
		
		Notice that in the non-standard characterization of continuity the points of continuity $r$ were required to be in $\mathr$, whereas in uniform continuity they are allowed to be in the non-standard extension $\starr$.
	\end{remark}

	\begin{examples}$ $
		\begin{quote}
		\begin{description}

			\item[(i)]Let $f: \rplus \to \mathr$ be defined by $f(x) = \frac{1}{x}$. By part \textbf{(ii)} of Example \ref{E: Cont} $f(x)$ is continuous on $\rplus$. However, it is \emph{not} uniformly continuous on $\rplus$. Indeed, let $r = \rho$ and $x = \rho^2$ where $\rho = \bra \varep \ket$ as in Example \ref{E: Canonical Infinitesimal}. Certainly they are both in $\starr_+$ so they are in the domain of $\starf(x) = \frac{1}{x}$. Furthermore, $\rho \approx \rho^2$ but $| \starf(\rho) - \starf(\rho^2) | = |\frac{1}{\rho} - \frac{1}{\rho^2} | = \frac{1}{\rho^2} - \frac{1}{\rho} = \frac{1- \rho}{\rho^2}$ is not infinitesimal since $1 - \rho$ is finite but not infinitesimal and $\frac{1}{\rho^2}$ is infinitely large.
			\item[(ii)] Let $f: \mathr \to \mathr$ be defined by $f(x) = \sin x$. Then $f(x)$ is uniformly continuous on $\mathr$. Indeed, let $x, r \in \starr$ such that $x \approx r$. We must show that $\star\sin x \approx \star\sin r$. We have $| \star\sin x - \star\sin r | = 2 \star\sin \frac{x - r}{2} \star\cos \frac{x + r}{2}$ using trigonometric identities. As $\frac{x-r}{2} \approx 0$, by part \textbf{(iii)} of Example \ref{E: Lim} we have $\star\sin \frac{x-r}{2} \approx 0$. Also, we have $\star\cos \frac{x+r}{2}$ is a finite number since $\star\cos x$ is bounded (Examples \ref{E: NSF} part \textbf{(iii)}), which gives that $| \star\sin x - \star\sin r | \approx 0$.
			\item[(iii)] Let $f: \mathr \to \mathr$ be defined by $f(x) = e^x$. It is well known that $f(x)$ is continuous on $\mathr$, however, it is not uniformly so. Indeed, select $\frac{1}{\rho}$ and $\frac{1}{\rho} + \rho$ in $\scl(\starr)$. We have $\frac{1}{\rho} \approx \frac{1}{\rho} + \rho$ but $\starf\left(\frac{1}{\rho} + \rho\right) - \starf\left(\frac{1}{\rho}\right) = e^{\frac{1}{\rho}}\left(e^\rho - 1\right) > e^{\frac{1}{\rho}} \frac{\rho}{2} \not\approx 0$. The inequality is justified via the asymptotic expansion of the Taylor Series for $e^\rho - 1$. Also, $e^{\frac{1}{\rho}} \frac{\rho}{2} \not\approx 0$ is actually infinitely large since $$e^{\frac{1}{\rho}} \frac{\rho}{2} \approx \st\left[e^{\frac{1}{\rho}} \frac{\rho}{2} \right] = \frac{1}{2} \st \left[ e^{\frac{1}{\rho}} \rho \right] = \frac{1}{2} \lim_{x \to 0_+} e^{\frac{1}{x}}x = \frac{1}{2} \lim_{x \to \infty} \frac{e^x}{x} = \frac{1}{2} \lim_{x \to \infty} \frac{e^x}{1} = \infty,$$ where the penultimate step is accomplished by l'Hospital's rule.
			
		\end{description}
		\end{quote}
	\end{examples}
	
	\begin{theorem}
		Let $X \subseteq \mathr$ be a compact set and $f: X \to \mathc$. If $f$ is continuous on $X$, then $f$ is uniformly continuous on $X$.
	\end{theorem}
	\begin{proof}
		As $X$ is compact we have $\starx \subseteq \mu(X)$ by Theorem \ref{T: NSCoS}. Since $\st(\mu(X)) = X$ and $X \subseteq \st(\starx)$ trivially we have $\st(\starx) = X$. Suppose $x, y \in \starx$ such that $x \approx y$. Then $c =: \st(x) = \st(y) \in X$. Therefore $\starf(x) \approx f(c) \approx \starf(y)$ since $f$ assumed continuous at $c$. Therefore $f$ is uniformly continuous on $X$. 
	\end{proof}

\section{Derivatives}\label{S: Derivatives}
	Let $X \subseteq \mathr$ and $c \in \mathr$ be a cluster point of $X$. Recall that a function $f: X \to \mathc$ is \emph{differentiable at $c$} with derivative $L \in \mathc$ if, by definition, for all $\varep \in \rplus$ there exists $\delta \in \rplus$ such that if $x \in \mathr$ is such that $0 < | x - c | < \delta$, then $$\left | \frac{f(x) - f(c)}{x - c} - L \right | < \varep.$$
	
	\begin{theorem}[Characterization of Derivatives]\label{T: NSD}\index{Derivative}
		Let $X \subseteq \mathr$, and consider the function $f: X \to \mathc$. The following are equivalent:
		\begin{quote}
			\begin{description}
			
				\item[(a)] The function $f$ is differentiable at $c \in \mathr$
				\item[(b)] There is a number $L \in \mathc$ such that for all $x \approx c$, $x \not= c$ we have: $$\frac{\starf(x) - \starf(c)}{x - c} \approx L.$$ Furthermore, when $L$ exists we have $f'(c) = L$.
				\item[(c)] $\st \left ( \frac{\starf(c + dx) - \starf(c)}{dx} \right ) = L$ for all non-zero infinitesimals $dx$ such that $c + dx \in \starx$.
	
			\end{description}
		\end{quote}
	\end{theorem}
	\begin{proof} We prove the equivalence of \textbf{(a)} and \textbf{(b)}.
	\begin{description}
		
		\item[($\Rightarrow$)] Suppose $f : X \to \mathc$ is differentiable at $c \in \mathr$ with derivative $L = f'(c) \in \mathc$. Then $\lim_{x \to c} \frac{f(x) - f(c)}{x - c} = L.$ By part \textbf{(c)} of Theorem \ref{T: NSL} we have $\frac{\starf(x) - \starf(c)}{ x - c } \approx L = f'(c)$ for all $x \in \starr$ such that $x \approx c$.
		\item[($\Leftarrow$)] Suppose there exists $L \in \mathc$ such that $\frac{\starf(x) - \starf(c)}{x-c} \approx L = f'(c)$ for all $x \approx c$, $x \not= c$. By Theorem \ref{T: NSL} we have that $\lim_{x \to c}\frac{f(x) - f(c)}{x -c} = L = f'(c)$.
		
	\end{description}
	
	As for the equivalence of \textbf{(b)} and \textbf{(c)} we refer to the equivalence of \textbf{(b)} and \textbf{(d)} from Theorem \ref{T: NSL}. 
	\end{proof}
	
	\begin{corollary}
		Let $X \subseteq \mathr$ and consider $f: X \to \mathc$. Then $f'(x) = \st \left ( \frac{ f(x + dx) - f(x) } {dx} \right )$ for all $dx \approx 0$, $dx \not= 0$ such that $x + dx \in \starx$.
	\end{corollary}
	
	\begin{remark}[Reduction of Quantifiers]\index{Reduction of Quantifiers}
		Recall for $f: X \to \mathc$, $L \in \mathc$, and $c \in \mathr$ a cluster point of $X$, $L$ is the derivative of $f(x)$ at $c$ if:
		\begin{equation}
			(\forall \varep \in \rplus)(\exists \delta \in \rplus)(\forall x \in X)\left [0 < |x - c| < \delta \Rightarrow \left |\frac{f(x)-f(c)}{x-c} - L \right| < \varep \right ]
		\end{equation}
		The formalization from Theorem \ref{T: NSD} is:
		\begin{equation}
			(\forall x \in \starx)\left [ x \approx c \Rightarrow \frac{\starf(x) - \starf(c)}{x - c} \approx L \right ]
		\end{equation}
		
		As with limits and continuity before, we see that the non-standard characterization has two fewer quantifiers and that those quantifiers commute (though in this case trivially).
	\end{remark}
	
	\begin{examples}[Computations]$ $
	\begin{quote}
		\begin{description}
		
			\item{(i)} Let $f(x) = x^3$. Let $dx \approx 0$, $dx \not= 0$ such that $x + dx \in \starx$ and apply the above corollary so that:
				\begin{align}
					f'(x) &= \st \left ( \frac{(x + dx)^3 - x^3}{dx} \right )\notag\\ 
						&= \st \left ( \frac{ x^3 + 3x^2 dx + 3x dx^2 + dx^3 - x^3}{dx} \right )\notag\\
						&= \st \left ( \frac{dx ( 3x^2 + 3xdx + dx^2)}{dx} \right )\notag\\
						&= \st(3x^2 + 3xdx + dx^2) = 3x^2.\notag
				\end{align}
			\item[(ii)] (Power Rule) Let $f(x) = x^n$. Let $dx \approx 0$, $dx \not= 0$ such that $x + dx \in \starx$ and apply the above corollary so that:
		\begin{align}
			f'(x) &= \st \left ( \frac{(x+dx)^n - x^n}{dx} \right )\notag\\
				&= \st \left ( \frac{\sum_{i=0}^{n} {n \choose i} x^{n-i} dx^{i} - x^n}{dx} \right )\notag\\
				&= \st \left ( \frac{\sum_{i=1}^{n} {n \choose i} x^{n-i} dx^{i}}{dx} \right )\notag\\
				&= \st \left ( \frac{ dx ( \sum_{i=1}^{n} {n \choose i} x^{n-i} dx^{i-1} )}{dx} \right )\notag\\
				&= \st \left ( \sum_{i=1}^{n} {n \choose i} x^{n-i} dx^{i-1} \right )\notag\\
				&= \st \left ( {n \choose 1} x^{n-1} + \sum_{i = 2}^{n} {n \choose i}x^i dx^{i - 1} \right )\notag\\
				&= \st (n x^{n - 1} ) + \st \left ( \sum_{i = 1}^{n} {n \choose i} x^{n- i} dx^{i - 1} \right ) = n x^{n - 1}.\notag
		\end{align}
			\item[(iii)] Let $f(x) = \sin(x)$. Let $dx \approx 0$, $dx \not= 0$ such that $x + dx \in \starx$ and apply the above corollary so that:
				\begin{align}
					f'(x) &= \st \left ( \frac{ \sin(x + dx) - \sin(x)}{dx} \right )\notag\\
						&= \st \left ( \frac{\sin(x)\cos(dx) + \cos(x)\sin(dx) - \sin(x)}{dx} \right )\notag\\
						&= \st \left ( \sin(x)\frac{(\cos(dx) - 1)}{dx} + \cos(x)\frac{\sin(dx)}{dx} \right )\notag
				\end{align}
				Now, $\st \left ( \frac{\sin(dx)}{dx} \right ) = \lim_{x \to 0} \frac{\sin(x)}{x} = 1$ and $\st \left ( \frac{\cos(dx) - 1}{dx} \right ) = \lim_{x \to 0} \frac{ \cos(x) - 1}{x} = 0$, both by l'Hospital's rule. Continuing the calculations on $f'(x)$ we have: $$\st \left ( \sin(x)\frac{(\cos(dx) - 1)}{dx} + \cos(x)\frac{\sin(dx)}{dx} \right ) = 0 + \cos(x) = \cos(x).$$
			
		\end{description}
	\end{quote}
	\end{examples}
	
	We now demonstrate the ease with which this characterization allows us to prove the usually burdensome chain rule.
	
	\begin{corollary}[Chain Rule]\label{C: Chain}
		Let $X \subseteq \mathr$ and $Y \subseteq g[X]$. If $g: X \to \mathc$ is differentiable at $c \in \mathr$ and $f: Y \to \mathc$ is differentiable at $g(c)$, then $f \circ g$ is differentiable at $c$ and $(f \circ g)'(c) = f'(g(c))g'(c)$.
	\end{corollary}
	\begin{proof}
		We show that for all $x \approx c$ with $x \not= c$, $$\frac{\starf(\starg(x)) - \starf(\starg(c))}{x - c} \approx f'(g(c))g'(c).$$ Indeed, let $x \approx c$. If $\starg(x) = \starg(c)$ then we are done. In the case that they are not equal, by Theorem \ref{T: NSD}, we have $$\frac{\starf(\starg(x)) - \starf(\starg(c))}{x - c} = \left (\frac{\starf(\starg(x)) - \starf(\starg(c))}{\starg(x) - \starg(c)} \right ) \left (\frac{\starg(x) - \starg(c)}{x - c}\right ) \approx f'(g(c))g'(c).$$
	\end{proof}
		
\section{Sequences of Functions}\label{S: Sequences of Functions}
	We now turn our attention to sequences of functions, and we explore what we may say about them in the language of non-standard analysis. We recall what it means for a sequence of functions to converge in standard analysis.

	Let $X \subseteq \mathr$, $S \subseteq X$, $f_n, f : X \to \mathc$, and let $(f_n)$ be a sequence of functions. Recall that $(f_n)$ converges pointwisely to $f$ on $S$ if (by definition),
	\begin{equation}\label{E: SequenceCont}
		(\forall x \in S)(\forall \varep \in \rplus)(\exists K \in \mathn)(\forall n \in \mathn)[n > K \Rightarrow | f_n(x) - f(x) | < \varep].
	\end{equation}
	Also recall that $(f_n)$ converges to $f$ \emph{uniformly} ($(f_n) \rightrightarrows f$) on $S$ if (by definition),
	\begin{equation}\label{E: SequenceUCont}
		(\forall \varep \in \rplus)(\exists K \in \mathn)(\forall x \in S)(\forall n \in \mathn)[n > K \Rightarrow | f_n(x) - f(x) | < \varep].
	\end{equation}

	\begin{theorem}[Convergence of a Sequence of Functions]\label{T: NSCSF}\index{Limit!of a Sequence of Functions}
		Let $X \subseteq \mathr$, $S \subseteq X$, and suppose $f_n, f: X \to \mathc$ with $(f_n)$ a sequence of functions. Then the following are equivalent:
		\begin{quote}
			\begin{description}
				
				\item[(a)] $(f_n)$ converges to $f$ on $S$.
				\item[(b)] $(\forall x \in S)(\forall n \in \starn \setminus \mathn)[\starf_n(x) \approx f(x)]$.
				
			\end{description}
		\end{quote}
	\end{theorem}
	\begin{proof}
		We have $(f_n) \to f$ on $S$ \iff $(\forall x \in S)[\lim_{n \to \infty} f_n(x) = f(x)]$ \iff (by Corollary \ref{C: NSLS}) $(\forall x \in S)(\forall n \in \starn \setminus \mathn)[\starf_n(x) \approx f(x)]$.
	\end{proof}
	
	\begin{theorem}[Uniform Convergence of a Sequence of Functions]\label{T: NSUCSF}\index{Uniform Convergence of a Sequence of Functions}
		Let $X \subseteq \mathr$, $S \subseteq X$, and suppose $f_n, f: X \to \mathc$ with $(f_n)$ a sequence of functions. Then the following are equivalent:
		\begin{quote}
			\begin{description}
				
				\item[(a)] $(f_n)$ converges uniformly to $f$ on $S$.
				\item[(b)] $(\forall x \in \stars)(\forall n \in \starn \setminus \mathn)[\starf_n(x) \approx \starf(x)]$.
				
			\end{description}
		\end{quote}
	\end{theorem}
	\begin{proof}$ $
		\begin{description}
		
			\item[($\Rightarrow$)] Let $(f_n) \to f$ uniformly on $S$. So, (\ref{E: SequenceUCont}) holds. Fix $\varep \in \rplus$. Then there exists $K \in \rplus$ such that $(\forall x \in S)(\forall n \in \mathn)[n > K \Rightarrow |f_n(x) - f(x)| < \varep].$ By the Transfer Principle we obtain, $(\forall x \in \stars)(\forall x \in \starn)[n > K \Rightarrow |\starf_n(x) - \starf(x)| < \varep].$ Certainly if $n \in \starn \setminus \mathn$ then $n > K$ for any $K \in \rplus$. Thus, $(\forall x \in \stars)(\forall n \in \starn \setminus \mathn)[|\starf_n(x) - \starf(x)| < \varep],$ which gives $(\forall x \in \stars)(\forall n \in \starn \setminus \mathn)[\starf_n(x) \approx \starf(x)]$, as desired.
			\item[($\Leftarrow$)] Assume \textbf{(b)}. If $\varep \in \rplus$ (arbitrarily fixed), then $(\forall x \in \stars)(\forall n \in \starn \setminus \mathn)[|\starf_n(x) - \starf(x)| < \varep]$. So trivially, $(\exists K \in \starn)(\forall x \in \stars)(\forall n \in \starn)[n > K \Rightarrow |\starf_n(x) - \starf(x)| < \varep]$ (For instance pick $K \in \starn \setminus \mathn$). Apply the Transfer Principle so that $(\exists K \in \mathn)(\forall x \in S)(\forall n \in \mathn)[n > K \Rightarrow |f_n(x) - f(x)| < \varep]$. As $\varep \in \rplus$ was arbitrary, we have $(f_n) \rightrightarrows f$. 
		
		\end{description}
	\end{proof}
		
	\begin{remark}[Reduction of Quantifiers]\index{Reduction of Quantifiers}
		Comparing (\ref{E: SequenceCont}) with \textbf{(b)} from Theorem \ref{T: NSCSF}, and (\ref{E: SequenceUCont}) with \textbf{(b)} from Theorem \ref{T: NSUCSF} we note the now common result. There are two less quantifiers in the non-standard characterizations, and these quantifiers commute.
	\end{remark}
	
	\begin{theorem}
		Suppose $(f_n)$ is a sequence of continuous functions on $X \subseteq \mathr$ for all $n \in \mathn$, and that $(f_n) \rightrightarrows f$ on $X$. Then $f$ is continuous on $X$.
	\end{theorem}
	\begin{proof}
		Let $x \in \starx$ and $c \in X$ with $x \approx c$. Our goal is to show $\starf(x) \approx f(c)$. By assumption, for all $n \in \mathn$, $\starf_n(x) \approx f_n(c)$. 
		
		We claim there exists $\nu \in \starn \setminus \mathn$ such that $\starf_\nu(x) \approx \starf_\nu(c)$. Indeed, consider $A =: \{ n \in \starn : | \starf_n(x) - \starf_n(c)| < \frac{1}{n} \}$. We know $A \not= \varnothing$ since $\mathn \subseteq A$ (by assumption). So, $A$ contains arbitrarily large finite numbers, and by the Overflow Principle (Corollary \ref{C: Spilling Principles}), $A$ contains an infinitely large $\nu$.
		
		Having established $\nu \in \starn \setminus \mathn$ such that $\starf_\nu(x) \approx \starf_\nu(c)$ we have $$|\starf(x) - f(c)| \leq |\starf(x) - \starf_\nu(x)| + |\starf_\nu(x) - \starf_\nu(c) + |\starf_\nu(c) - f(c)|.$$ The first and third summands are infinitesimal since $\nu \in \starn \setminus \mathn$ and $(f_n) \rightrightarrows f$, and the second summand is infinitesimal by the above claim. Hence, $|\starf(x) - \starf(c)|$ is infinitesimal. 
	\end{proof}
		
	\begin{examples}$ $
	\begin{quote}
		\begin{description}
		
			\item[(i)] Consider $(f_n) = (x^n)$. Notice that $f_n(x)$ converges to $0$ on $[0, 1)$. We first show that $(f_n)$ is not uniformly convergent on $[0,1)$. Let $\nu \in \starn \setminus \mathn$, and define $x =: 1 - \frac{1}{\nu} \in {\star[0,1)}$. Then $x^\nu = (1 - \frac{1}{\nu})^\nu \approx \lim_{n \to \infty} (1 - \frac{1}{n})^n = e^{-1} \not= 0$. Therefore, $x^\nu \not\approx 0$ and so $(f_n)$ is not uniformly convergent on $[0, 1)$.
			\item[(ii)] Consider $(f_n) = (x^n)$, on the interval $[0, \delta)$ for any $0 < \delta < 1$. We claim that $(f_n) \rightrightarrows 0$ on $[0, \delta)$. Note ${\star[0, \delta)} = \{ x \in \starr : 0 \leq x < \delta\}$ where $\delta$ is standard. So $0 \leq x < \delta < 1$ giving $0 \leq x^n < \delta^n$, hence $0 \leq x^{\nu} < \delta^\nu \approx 0$. Therefore, $x^\nu \approx 0$. Giving $x^n \rightrightarrows 0$ on $[0, \delta)$.
			\item[(iii)] Consider $(f_n) = (\frac{1}{n}\sin(nx))$. We show that $(f_n)$ is uniformly convergent to $0$ on $\mathr$. Note that $|{^*\sin(nx)}| \leq 1$ for all $n \in \starn$ and $x \in \starr$. Choose $x \in \starr$ and $n \in \starn \setminus \mathn$. Then $|\frac{1}{n}{^*\sin(nx)}| \leq \frac{1}{n} \approx 0$. As $x$ and $n$ were arbitrary, $(f_n) \rightrightarrows 0$ on $\mathr$ by Theorem \ref{T: NSUCSF}.
			\item[(iv)] Consider $(f_n) = (e^{-(x-n)^2})$. We illustrate the subtle difference between the convergence and uniform convergence of a sequence of functions.
			
			We first show that $(f_n)$ converges to $0$ on $\mathr$. Let $x \in \mathr$ and $\nu \in \starn \setminus \mathn$ be arbitrarily chosen and fixed. Then $e^{-(x-\nu)^2} \approx 0$ since $\nu$ is infinitely large while $x$ is finite, therefore $(x - \nu)$ is infinitely large. Squaring it only improves our situation.
			
			We now show that $(f_n)$ is not uniformly convergent to $0$ on $\mathr$. Let $\nu \in \starn \setminus \mathn$ be arbitrarily chosen and fixed. Then let $x = \nu \in \starr$ so that $e^{-(x - \nu)^2} = e^{-(\nu - \nu)^2} = e^0 = 1 \not\approx 0$. Therefore, $(f_n)$ is not uniformly convergent to $0$ on $\mathr$.
			
			The key to ordinary convergence was that $x$ was only allowed to be \emph{finite}, while in uniform convergence $x$ had to be non-standard. Therefore, $x$ could, in effect, keep up with $\nu \in \starn \setminus \mathn$.
			\item[(v)] Consider $(f_n) = (\frac{1}{n}e^{-(x-n)^2})$. We show that $(f_n)$ is uniformly convergent to $0$ on $\mathr$. We know that for all $x \in \starr$ and for all $\nu \in \starn \setminus \mathn$, $e^{-(x-\nu)^2} \leq 1$. Now, let $x \in \starr$ and $\nu \in \starn \setminus \mathn$ be arbitrarily chosen and fixed. Then $\frac{1}{\nu}e^{-(x-\nu)^2} \leq \frac{1}{\nu} \approx 0$. So by Theorem \ref{T: NSUCSF}, $(f_n) \rightrightarrows 0$ on $\mathr$.
					
		\end{description}
	\end{quote}
	\end{examples}

	\printindex
	
\end{document}